\documentclass[12pt]{amsart}
\usepackage{a4wide,enumerate,xcolor}
\usepackage{amsmath,comment,hyperref}
\allowdisplaybreaks

\let\pa\partial
\let\na\nabla
\let\eps\varepsilon
\newcommand{\N}{{\mathbb N}}
\newcommand{\R}{{\mathbb R}}
\newcommand{\T}{{\mathbb T}}
\newcommand{\diver}{\operatorname{div}}

\newtheorem{theorem}{Theorem}
\newtheorem{lemma}[theorem]{Lemma}

\newtheorem{remark}[theorem]{Remark}
\newtheorem{corollary}[theorem]{Corollary}
\newtheorem{definition}{Definition}

\begin{document}

\title[Relaxation approximation of the Busenberg--Travis
system]{Fluid relaxation approximation of the \\ Busenberg--Travis
cross-diffusion system  
}

\author[J. A. Carrillo]{Jos\'e Antonio Carrillo}
\address{Mathematical Institute, University of Oxford, Oxford, OX2 66G, UK}
\email{carrillo@maths.ox.ac.uk}

\author[X. Chen]{Xiuqing Chen}
\address{School of Mathematics, Sun Yat-Sen University, Zhuhai 519082, Guangdong Province, China}
\email{chenxiuqing@mail.sysu.edu.cn}

\author[B. Du]{Bang Du}
\address{School of Mathematics, Sun Yat-Sen University, Zhuhai 519082, Guangdong Province, China}
\email{dubang@mail2.sysu.edu.cn}

\author[A. J\"ungel]{Ansgar J\"ungel}
\address{Institute of Analysis and Scientific Computing, TU Wien, Wiedner Hauptstra\ss e 8--10, 1040 Wien, Austria}
\email{juengel@tuwien.ac.at}

\date{\today}

\thanks{The first author was partially supported by by the EPSRC grants no.~EP/V051121/1 and by the ``Maria de Maeztu'' Excellence Unit IMAG, reference CEX2020-001105-M, funded by MCIN/AEI/10.13039/501100011033. The second and third authors acknowledge support from the National Natural Science Foundation of China (NSFC),  grant 12471206. The last author acknowledges partial support from the Austrian Science Fund (FWF), grant DOI 10.55776/P33010 and 10.55776/F65. This work has received funding from the European Research Council (ERC) under the European Union's Horizon 2020 research and innovation programme, ERC Advanced Grants Nonlocal-CPD, no.~883363, and NEUROMORPH, no.~101018153.}

\begin{abstract}
The Busenberg--Travis cross-diffusion system for segregating populations is approximated by the compressible Navier--Stokes--Korteweg equations on the torus, including a density-dependent viscosity and drag forces. The Korteweg term can be associated to the quantum Bohm potential. The singular asymptotic limit is proved rigorously using compactness and relative entropy methods. The novelty is the derivation of energy and entropy inequalities, which reduce in the asymptotic limit to the Boltzmann--Shannon and Rao entropy inequalities, thus revealing the double entropy structure of the limiting Busenberg--Travis system.
\end{abstract}



\maketitle


\section{Introduction}

The aim of this paper is to analyze a fluiddynamical approximation of the Busenberg--Travis population cross-diffusion system \cite{BuTr83}
\begin{align}\label{1.BT}
  \pa_t\rho_i - \diver\big(k_i\rho_i\na(\rho_1+\rho_2)\big) = 0 \quad
  \mbox{in }\T^d,\ t>0,\ i=1,2,
\end{align}
where $\rho_i$ is the density of the $i$th population species, $k_i>0$ is a diffusion coefficient, $\T^d$ is the multidimensional torus, and we impose the initial conditions $\rho_i(0)=\rho_i^0$ in $\T^d$ for $i=1,2$. The equations have been suggested by Busenberg and Travis \cite{BuTr83} to describe the segregation of populations, see also \cite{BGHP85,GuPi84}. They have also been proposed, based on interacting particle systems, to introduce short-range repulsion in cell-cell adhesion models \cite{CMSTT19,MT15}.

\subsection{Motivation and model setting}

Our motivation for an approximation of \eqref{1.BT} is to recover the entropy structure of \eqref{1.BT} from the energy and entropy of the approximating fluid system. Indeed, it has been shown in \cite{JPZ22,JuZu20} that system \eqref{1.BT} possesses {\em two} entropies, the Boltzmann--Shannon entropy $H_1$ and the Rao entropy $H_2$,
\begin{align*}
  H_1 = \sum_{i=1}^2\int_{\T^d}k_i^{-1}\rho_i(\log\rho_i-1)dx,
  \quad H_2 = \frac12\int_{\T^d}(\rho_1+\rho_2)^2 dx.
\end{align*}
This means that formally, along solutions to \eqref{1.BT},
\begin{align*}
  \frac{dH_1}{dt} + \int_{\T^d}|\na(\rho_1+\rho_2)|^2 dx &= 0, \\
   \frac{dH_2}{dt} + \int_{\T^d}(k_1\rho_1+k_2\rho_2)
  |\na(\rho_1+\rho_2)|^2 dx &= 0.
\end{align*}
However, the origin of this double entropy structure remained unclear. We propose a fluiddynamical approximation that possesses the thermodynamical entropy $H_1$ and an energy containing $H_2$. Thus, the entropy structure of \eqref{1.BT} originates from the energy and entropy of the associated fluiddynamical system.

Before we make this statement precise, we comment on the cross-diffusion system \eqref{1.BT}. The diffusion matrix associated to \eqref{1.BT} has rank one, such that this system is of mixed hyperbolic--parabolic type. Indeed, we can reformulate \eqref{1.BT} as a diffusion equation for the total density $\rho_1+\rho_2$ and a transport equation for one of the densities:
\begin{align}
  & \pa_t(\rho_1+\rho_2)
  = \diver\big((k_1\rho_1+k_2\rho_2)\na(\rho_1+\rho_2)\big),
  \label{1.lim1} \\
  & \pa_t\rho_i+\diver(\rho_i\bar{u}_i)=0\quad\mbox{in }\T^d,\ t>0,\
  i=1,2, \label{1.lim2}
\end{align}
where $\bar{u}_i=-k_i\na(\rho_1+\rho_2)$ is the velocity associated to the $i$th species. Observe that the equations decouple if $k:=k_1=k_2$; then $\rho_1+\rho_2$ solves a porous-medium equation, and the densities $\rho_i$ are transported with the common velocity $\bar{u}=-k\na(\rho_1+\rho_2)$. We refer to \cite{DHJ23,DrJu20} for details on the hyperbolic--parabolic structure.

A fluiddynamical approximation of \eqref{1.BT} takes the form
\begin{equation}\label{1.S}
\begin{aligned}
  & \pa_t\rho_i + \diver(\rho_i u_i) = 0, \quad i=1,2, \\
  & \eps\pa_t(\rho_iu_i) + \eps\diver(\rho_iu_i\otimes u_i)
  = \eps\diver S - k_i^{-1}\rho_i u_i - \rho_i\na(\rho_1+\rho_2),
\end{aligned}
\end{equation}
where $u_i$ is the partial velocity of the $i$th species, $\eps>0$ is a small number, $S$ is the stress tensor, $-k_i^{-1}\rho_iu_i$ is the relaxation term, and $-\rho_i\na(\rho_1+\rho_2)$ is a force term. The formal limit $\eps\to 0$ in \eqref{1.S} leads to \eqref{1.lim1}--\eqref{1.lim2}. As this limit is singular, its rigorous proof is delicate.

The main difficulty comes from the force term $-\rho_i\na(\rho_1+\rho_2)$, since the energy of the fluiddynamical equations does not provide any gradient estimate. This issue does not occur in the relaxation-time limit of the Euler--Poisson equations, since the force reads as $-\rho_i\na\Phi$, and the electric potential $\Phi$ solves the Poisson equation, thus providing sufficient regularity to apply the div--curl lemma \cite{JuPe99}. The lack of a gradient bound can be overcome by allowing for a Korteweg term in \eqref{1.S}, leading to Euler--Korteweg equations \cite{LaTz13} or Navier--Stokes--Korteweg equations \cite{CaLa23}. More precisely, we add to the right-hand side of \eqref{1.S} the expression
\begin{align*}
  K = \eps\rho_i\na\bigg(\kappa(\rho_i)\Delta\rho_i
  + \frac12\kappa'(\rho_i)|\na\rho_i|^2\bigg),
\end{align*}
where $\kappa(\rho_i)$ is the capillarity coefficient. In this paper, we choose $\kappa(\rho_i)=1/(2\rho_i)$, which leads to
\begin{align*}
  K = \eps\rho_i\na\bigg(\frac{\Delta\sqrt{\rho_i}}{\sqrt{\rho_i}}\bigg).
\end{align*}
The expression $\Delta\sqrt{\rho_i}/\sqrt{\rho_i}$ is known in quantum mechanics as the Bohm potential, and equations \eqref{1.S} become the quantum Navier--Stokes equations studied in, e.g., \cite{ACLS21,Jue10,LaVa17,VaYu16}. Other choices for $\kappa(\rho_i)$ are discussed in Remark \ref{rem.gener}. As in \cite{Jue10,VaYu16}, we use the density-dependent stress tensor $S=\rho_i\na u_i$. This dependence is needed in the derivation of the entropy inequality.

A second difficulty is due to the fact that we control the kinetic energy $\rho_i|u_i|^2$ in $L^1(\T^d)$ only. This is not sufficient to prevent concentration phenomena in the convective term $\rho_i u_i\otimes u_i$. A way out is the introduction of additional drag forces like in \cite[Sec.~9]{BrDe07} and used in the context of the quantum Navier--Stokes equations in \cite{VaYu16}.

Summarizing, we perform the asymptotic limit $\eps\to 0$ in the compressible Navier--Stokes equations with Korteweg regularization and drag forces:
\begin{align}
  & \pa_t\rho_i + \diver(\rho_i u_i) = 0 \quad\mbox{in }\T^d,\ t>0,
  \ i=1,2, \label{1.mass} \\
  & \eps\pa_t(\rho_iu_i) + \eps\diver(\rho_iu_i\otimes u_i)
  = \eps\rho_i\na\bigg(
  \frac{\Delta\sqrt{\rho_i}}{\sqrt{\rho_i}}\bigg)
  + \eps\diver(\rho_i\na u_i) \label{1.mom} \\
  &\phantom{xx}
  -\eps u_i - \eps\rho_i u_i|u_i|^2 - k_i^{-1}\rho_iu_i
  - \rho_i\na(\rho_1+\rho_2), \nonumber
\end{align}
subject to the initial conditions
\begin{align}\label{1.ic}
  \rho_i(0) = \rho_i^0, \quad \rho_i(0)u_i(0) = \rho_i^0 u_i^0
  \quad\mbox{in }\T^d,\ i=1,2.
\end{align}
In the following, we set $\rho=(\rho_1,\rho_2)$ and $u=(u_1,u_2)$.

\subsection{Key ideas}\label{sec.ideas}

A priori estimates are derived by estimating the energy
\begin{align}\label{1.E}
  E(\rho,u) = \int_{\T^d}\bigg(\frac12(\rho_1+\rho_2)^2
  + \frac{\eps}{2}\sum_{i=1}^2\rho_i|u_i|^2
  + \eps\sum_{i=1}^2|\na\sqrt{\rho_i}|^2\bigg)dx,
\end{align}
which is the sum of the potential, kinetic, and Korteweg energies. A formal computation, made rigorous on the approximate level in Section \ref{sec.ener}, shows that
\begin{align*}
  \frac{dE}{dt} + \sum_{i=1}^2\int_{\T^d}\big(k_i^{-1}\rho_i|u_i|^2
  + \eps\rho_i|\na u_i|^2 + \eps|u_i|^2
  + \eps\rho_i|u_i|^4\big)dx = 0.
\end{align*}
Unfortunately, this equality does not provide any gradient bound for the densities independent of $\eps$. Our main idea is to obtain such a bound from the entropy
\begin{equation}\label{1.H}
  H(\rho) = \sum_{i=1}^2k_i^{-1}\int_{\T^d}\rho_i(\log\rho_i-1)dx.
\end{equation}
A formal computation (made rigorous for the approximate solutions in Section \ref{sec.ent}) shows that
\begin{align*}
  \frac{dH}{dt} + \int_{\T^d}|\na(\rho_1+\rho_2)|^2 dx
  + \frac{\eps}{2}\int_{\T^d}\rho_i|D^2\log\rho_i|^2 dx
  \le R,
\end{align*}
where the remainder $R$ depends on the unknowns $\rho_i$, $u_i$ and their derivatives (like $\sqrt{\rho_i}\na u_i$), but it is independent of $\eps$. The last integral on the left-hand side gives ($\eps$-dependent) estimates in $H^2(\T^d)$ and $W^{1,4}(\T^d)$ from the inequality
\begin{align}\label{1.ineq}
  \int_{\T^d}\rho_i|D^2\log\rho_i|^2 dx
  \ge c(d)\int_{\T^d}\big(|\Delta\sqrt{\rho_i}|^2
  + |\na\sqrt[4]{\rho_i}|^4\big) dx
\end{align}
for some $c(d)>0$, which holds for sufficiently smooth functions $\rho_i$; see \cite[Lemma 2.2]{JuMa08} and \cite[Appendix]{Jue10} (or \cite[Lemma 2.1]{VaYu16}) for a proof. Then the remainder $R$ can be controlled by the bounds coming from the energy and \eqref{1.ineq}. Since the limiting system \eqref{1.BT} may possess discontinuous solutions \cite{BDM10}, we cannot expect gradient bounds for the individual densities $\rho_i$ but only for the sum $\rho_1+\rho_2$. Thus, we cannot expect better estimates.

Denoting by $(\rho^\eps,u^\eps)$ a weak solution to \eqref{1.mass}--\eqref{1.ic}, the energy and entropy estimates together with the Aubin--Lions lemma yield strong convergence of the sum $\bar\rho^\eps:=\rho_1^\eps+\rho_2^\eps$, but we have only weak convergence for $\rho_i^\eps$ and $\na\bar\rho^\eps$. Thus, the limit in the product $\rho_i^\eps\na\bar\rho^\eps$ cannot be easily identified.

We show two results. First, we prove that the strong limit $\bar\rho$ of $\bar\rho^\eps$ solves \eqref{1.lim1} with a ``defect'',
\begin{align*}
  \pa_t\bar\rho - \diver\big((k_1\rho_1+k_2\rho_2)\na\bar\rho\big)
  = (k_2-k_1)(k_2^{-1}J_2+\rho_2\na\bar\rho)
\end{align*}
where $\rho_i$ is the weak limit of $\rho_i^\eps$ and $J_2$ is the weak limit of $\rho_2^\eps u_2^\eps$ (see Theorem \ref{thm.eps}). The right-hand side vanishes if $k_1=k_2$ or if we can identify $J_2$ with $-k_2\rho_2\na\bar\rho$. Unfortunately, we have not been able to prove this identification for general values of $k_1$, $k_2$. Indeed, neither the div--curl lemma nor Feireisl's viscous flux approach can be applied because of the lack of suitable gradient bounds uniform in $\eps$.

Second, we consider the case $k_1=k_2=1$. Then $\bar\rho$ solves the quadratic porous-medium equation (see \eqref{1.lim1}), but we still have no information about the evolution of $\rho_i$. To derive the dynamics, we apply the relative entropy method. The idea is to compare a solution $(\rho^\eps,u^\eps)$ to \eqref{1.mass}--\eqref{1.mom} with a solution $(\bar\rho,\bar{u})$ to the limit system \eqref{1.lim1}--\eqref{1.lim2} with $k_1=k_2=1$ and $\bar{u}=-\na\bar\rho$. The relative entropy (more precisely: relative energy) is defined by
\begin{align*}
  E_R(\rho^\eps,u^\eps|\bar\rho,\bar{u})
  = \int_{\T^d}\bigg\{\frac12(\bar\rho^\eps-\bar\rho)^2
  + \eps\sum_{i=1}^2\bigg(\frac{\rho_i^\eps}{2}|u_i^\eps-\bar{u}|^2
  + |\na(\rho_i^\eps)^{1/2}|^2\bigg)\bigg\}dx.
\end{align*}
A computation, made rigorous in Section \ref{sec.eps2}, shows that
\begin{align*}
  \frac{dE_R}{dt}(\rho^\eps,u^\eps|\bar\rho,\bar{u})
  + \sum_{i=1}^2\int_{\T^d}\rho_i^\eps|u_i^\eps-\bar{u}|^2 dx
  \le C\sqrt[4]{\eps}.
\end{align*}
Here, we need the assumption $k_1=k_2$. We infer that $(\rho_i^\eps)^{1/2}(u_i^\eps-\bar{u})\to 0$ strongly in $L^2(0,T;L^2(\T^d))$, which implies that $\rho_i^\eps u_i^\eps \rightharpoonup \rho_i\bar{u}$ weakly in $L^2(0,T;L^{4/3}(\T^d))$. This allows us to identify $J_i$ with $\rho_i\bar{u}=-\rho_i\na\bar\rho$, and $\rho_i$ solves the transport equation \eqref{1.lim2}.

\subsection{State of the art}

There are many results in the literature on the relaxation limit in hyperbolic systems. General results can be found, for instance, in \cite{DoMa04}. The relaxation limit in Euler--Poisson systems, leading to the drift-diffusion equations, was proved in \cite{JuPe99}, exploiting the regularizing effect of the Poisson equation. Relaxation-time limits were also performed in Euler--Maxwell \cite{PWG11} and quantum hydrodynamic equations \cite{JLM06}. Using compactness methods, relaxation limits in the compressible Navier--Stokes--Poisson equations \cite{KoOg13} and in the quantum Navier--Stokes equations \cite{ACLS21} were proved. Note that our limit is more delicate since the regularizing terms vanish in the limit.

The relative entropy method was first used by Dafermos \cite{Daf79} and Di Perna \cite{DiP79}. It was extended later by Lattanzio and Tzavaras \cite{LaTz13} to compare the solution to the frictional Euler equations with the solution to the porous-medium equation. This technique was also applied in the analysis of the high-friction regime of Euler--Korteweg equations \cite{HJT19}, for more general aggregation-diffusion equations \cite{CPW}, compressible Navier--Stokes--Korteweg systems \cite{CaLa23}, and Euler--Riesz models \cite{ACC,AGT}.

The Busenberg--Travis system \eqref{1.BT} can be derived from interacting particle systems in the mean-field limit, even for an arbitrary number of species \cite{CDJ19} and for nonlinear pressures \cite{CaGu24}. In the general case, the limiting system reads as
\begin{align}\label{1.gBT}
  \pa_t u_i = \diver(u_i\na p_i(u)), \quad
  p_i(u) = \sum_{j=1}^n a_{ij}u_j\quad\mbox{in }\T^d,\ i=1,\ldots,n,
\end{align}
where $a_{ij}\ge 0$ are some numbers. If the matrix $(a_{ij})$ is positive definite, a global existence analysis can be found in \cite{JPZ22}. If the matrix $(a_{ij})$ is of rank one, i.e.\ $a_{ij}=k_i$ like in our situation, the existence of global classical solutions to \eqref{1.BT} (under the positivity assumption $\sum_{i=1}^n\rho_i^0\ge c>0$) was proved in \cite{DrJu20}. The positivity ensures the regularity of the total density. If the positivity assumption is relaxed to nonnegativity, the existence of global measure-valued solutions can be shown \cite{HoJu23}. Steady states may be discontinuous \cite{CHS18}, and there is some numerical evidence \cite{BDM10} that this may be also true for transient solutions. The lack of regularity motivated the authors of \cite{CFSS18} to work in the one-dimensional setting with solutions of bounded variation.

The limit in \eqref{1.mass}--\eqref{1.mom} seems to be new, and we believe that it contributes to the understanding of the entropy structure of the Busenberg--Travis cross-diffusion system and possibly of related models.

\subsection{Main results}

We first define our notion of weak solution. This is necessary since the Korteweg term in \eqref{1.mom} needs special care.

\begin{definition}
We call $(\rho,u)$ with $\rho=(\rho_1,\rho_2)$ and $u=(u_1,u_2)$ a {\em weak solution} to \eqref{1.mass}--\eqref{1.ic} on $(0,T)$ if, for $i=1,2$, $\rho_i\ge 0$ in $\T^d\times(0,T)$,
\begin{align*}
  & \rho_i\in L^\infty(0,T;L^2(\T^d)), \quad
  \sqrt{\rho_i}u_i\in L^\infty(0,T;L^2(\T^d)), \\
  &\sqrt{\rho_i}\in L^2(0,T;H^2(\T^d)), \quad
  \rho_1+\rho_2\in L^2(0,T;H^1(\T^d)), \\
  & \sqrt{\rho_i}|\na u_i|\in L^2(0,T;L^2(\T^d)), \quad
  \sqrt[4]{\rho_i}|u_i|\in L^4(0,T;L^4(\T^d)),
\end{align*}
for all
$\phi\in C_0^\infty(\T^d\times[0,T))$ and $\psi\in C_0^\infty(\T^d\times[0,T);\R^d)$,
\begin{align}\label{1.massw}
  0 &= -\int_0^T\int_{\T^d}\rho_i\pa_t\phi dxdt
  - \int_{\T^d}\rho_i^0\phi(0)dx
  - \int_0^T\int_{\T^d}\rho_i u_i\cdot\na\phi dxdt, \\
  0 &= -\eps\int_0^T\int_{\T^d}\rho_i u_i\cdot\pa_t\psi dxdt
  - \eps\int_{\T^d}\rho_i^0 u_i^0\cdot\psi(0)dx
  - \eps\int_0^T\int_{\T^d}\rho_i(u_i\otimes u_i):\na\psi dxdt
  \label{1.momw} \\
  &\phantom{xx}+ \eps\int_0^T\int_{\T^d}\rho_i\na u_i:\na\psi dxdt
  + \eps\int_0^T\int_{\T^d}\Delta\sqrt{\rho_i}\big(2\na\sqrt{\rho_i}
  \cdot\psi + \sqrt{\rho_i}\diver\psi\big)dxdt \nonumber \\
  &\phantom{xx}
  + \eps\int_0^T\int_{\T^d}(u_i + \rho_i|u_i|^2u_i)\cdot\psi dxdt
  + \int_0^T\int_{\T^d}
  \big(k_i^{-1}\rho_i u_i + \rho_i\na(\rho_1+\rho_2)\big)
  \cdot\psi dxdt, \nonumber
\end{align}
and the initial conditions \eqref{1.ic} hold in the sense of distributions.
\end{definition}

We impose the following assumptions:
\begin{itemize}
\item[(A1)] Parameter: $d=1,2,3$, $T>0$, and $k_1>0$, $k_2>0$.
\item[(A2)] Initial data: $\rho_i^0\in L^2(\T^d)$, $\sqrt{\rho_i^0}|u_i^0|\in L^2(\T^d)$, $\sqrt{\rho_i^0}\in H^1(\T^d)$, $\log\rho_i^0\in L^1(\T^d)$ for $i=1,2$.
\end{itemize}
The assumption of at most $d=3$ space dimensions is due to Sobolev embeddings (we need $H^2(\T^d)\hookrightarrow L^\infty(\T^d)$ in Lemma \ref{lem.space}). Our analysis strongly depends on inequality \eqref{1.ineq}. To avoid boundary integrals, we consider equations \eqref{1.mass}--\eqref{1.mom} on the torus. The regularity of the initial data is needed to obtain a finite initial energy and entropy. We may allow for reaction terms in \eqref{1.mass} if the reactions depend nonlinearly on the total density $\rho_1+\rho_2$ only (since we have strong convergence only for the sum). We discuss further generalizations of the nonlinearities in Remark \ref{rem.gener}.

Our first main result reads as follows.

\begin{theorem}[Existence of solutions]\label{thm.ex}
Let Assumptions (A1)--(A2) hold. Then there exists a weak solution $(\rho,u)$ to \eqref{1.mass}--\eqref{1.ic} with $\rho=(\rho_1,\rho_2)$ and $u=(u_1,u_2)$ such that
\begin{align}\label{1.Eineq}
   E(\rho(t),u(t)) & + \sum_{i=1}^2\int_0^t\int_{\T^d}
  k_i^{-1}\rho_i|u_i|^2 dxds \\
  &+ \eps\sum_{i=1}^2\int_0^t\int_{\T^d}
  \big(\rho_i|\na u_i|^2 + |u_i|^2 + \rho_i|u_i|^4\big)dxds
  \le E(\rho^0,u^0), \nonumber \\
  & {H(\rho(t))}
  + C_1(d)\eps \int_0^t\int_{\T^d}\big(|\Delta\sqrt{\rho_i}|^2
  + |\na\sqrt[4]{\rho_i}|^4\big)dxds
  \le{C_2(\rho^0,u^0)}, \label{1.Hineq}
\end{align}
where $C_1(d)>0$ is a constant only depending on the space dimension $d$, $C_2(\rho^0,u^0)>0$ depends on the initial data (but not on $\eps$), and we recall definitions \eqref{1.E} and \eqref{1.H} of $E$ and $H$. Moreover, we have the regularity
\begin{align*}
  \pa_t\rho_i\in L^2(0,T;L^2(\T^d)), \quad
  \pa_t(\rho_iu_i)\in L^{4/3}(0,T;H^s(\T^d)'), \quad s>d/2+1.
\end{align*}
Therefore, the initial condition for $\rho_i$ holds a.e.\ in $\T^d$, and the initial condition for $\rho_iu_i$ holds in $H^s(\T^d)'$.
\end{theorem}

The proof of the existence of solutions is based on an approximate scheme. We add the parabolic regularization $\delta\Delta\rho_i$ to the mass balance equation \eqref{1.mass} and the regularization $\delta\eps\diver(u_i\otimes\na\rho_i)$ to the momentum balance equation (similarly as in the existence analysis for the compressible Navier--Stokes equations; see \cite[Sec.~7.2]{Fei04}). The latter term is needed to compensate some contributions coming from the parabolic regularization when deriving the energy inequality. The local existence of solutions is shown by the Faedo--Galerkin method. The solution can be extended to a global one thanks to the energy estimate. The entropy production in \eqref{1.Hineq} is obtained by applying inequality \eqref{1.ineq}.

If $k_1=k_2$, a computation similar to the proof of Lemma \ref{lem.HN} shows that the entropy inequality \eqref{1.Hineq} can be improved by replacing $C_2(\rho^0,u^0)$ by $H(\rho^0)+C\sqrt[4]{\eps}$, where $C>0$ is independent of $\eps$.

\begin{theorem}[Limit $\eps\to 0$]\label{thm.eps}
Let $(\rho^\eps,u^\eps)$ be a weak solution to \eqref{1.mass}--\eqref{1.ic} as constructed in Theorem \ref{thm.ex}. Then there exists a subsequence (not relabeled) such that for $i=1,2$,
\begin{align*}
  \rho_i^\eps\rightharpoonup\rho_i &\quad\mbox{weakly* in }
  L^\infty(0,T;L^2(\T^d)), \\
  \rho_1^\eps+\rho_2^\eps\to \rho_1+\rho_2 &\quad\mbox{strongly in }
  L^2(0,T;L^2(\T^d)), \\
  \rho_i^\eps u_i^\eps\rightharpoonup J_i &\quad\mbox{weakly in }
  L^2(0,T;L^{4/3}(\T^d)),
\end{align*}
and $\bar\rho:=\rho_1+\rho_2$ solves
\begin{align*}
  \pa_t\bar\rho - \diver\big((k_1\rho_1+k_2\rho_2)\na\bar\rho)\big)
  = -(k_2-k_1)\diver (k_2^{-1}J_2+\rho_2\na\bar\rho)
  \quad\mbox{in }\T^d,\ t>0,
\end{align*}
with the initial condition $\bar\rho(0)=\rho_1^0+\rho_2^0$ in the sense of $W^{1,4}(\T^d)'$.
\end{theorem}

The proof is based on the energy and entropy estimates proved in Theorem \ref{thm.ex} and on compactness arguments. If $k_1=k_2$, we can prove a stronger result, using the relative entropy method. Indeed, as explained in Section \ref{sec.ideas}, we are able to prove that not only $\rho_i^\eps u_i^\eps\rightharpoonup J_i$ weakly in $L^2(0,T;L^{4/3}(\T^d))$ (which follows from the energy inequality) but also $\rho_i^\eps u_i^\eps \rightharpoonup \rho_i\bar{u}$ weakly in $L^1(0,T;L^1(\T^d))$  
 (which follows from the relative entropy inequality). This allows us to identify $J_i=\rho_i\bar{u}=-\rho_i\na\bar\rho$.

\begin{theorem}[Limit $\eps\to 0$, $k_1=k_2$]\label{thm.eps2}
Let $k_1=k_2=1$, let $(\rho^\eps,u^\eps)$ be a weak solution to \eqref{1.mass}--\eqref{1.ic} as constructed in Theorem \ref{thm.ex}, and let $(\bar\rho,\bar{u})$ be the unique smooth solution to
\begin{align*}
  \pa_t\bar\rho + \diver(\bar\rho\bar{u}) = 0, \quad
  \bar{u} = -\na\bar\rho\quad\mbox{in }\T^d,\ t>0,
\end{align*}
with initial conditions $\bar\rho(0)=\rho_1^0+\rho_2^0$ (this requires smooth positive initial data). Then, as $\eps\to 0$,
\begin{align*}
  \rho_1^\eps+\rho_2^\eps \to \bar\rho &\quad\mbox{strongly in }
  L^2(0,T;L^2(\T^d)), \\
  \rho_i^\eps u_i^\eps \rightharpoonup \rho_i\bar{u} & \quad\mbox{weakly in }
      L^2(0,T;L^{4/3}(\T^d)).
\end{align*}
and the $L^2(\T^d)$-weak limit $\rho_i$ of $(\rho_i^\eps)$ solves the transport equation
\begin{align*}
  \pa_t\rho_i - \diver(\rho_i\na(\rho_1+\rho_2)) = 0 \quad
  \mbox{in }\T^d,\ t>0,\ i=1,2.
\end{align*}

\end{theorem}

\begin{remark}[Generalizations]\label{rem.gener}\rm
Our results can be extended in various directions. First, all results are valid for more than two species, and it is sufficient to replace the sum over $i=1,2$ by $i=1,\ldots,n$, where $n\in\N$ is arbitrary. Second, one may try more general Korteweg functions $\kappa(\rho_i)$. A simple choice is $\kappa(\rho_i)=1$, giving the higher-order regularization $K=\eps\rho_i\na\Delta\rho_i$, which equals the flux of the thin-film equation. The Korteweg energy density becomes $\eps|\na\rho_i|^2$ and the entropy production simplifies to $\eps\int_{\T^d}(\Delta\rho_i)^2 dx$, thus providing $H^2(\T^d)$ bounds for $\rho_i$. Since we do not need inequalities like \eqref{1.ineq} in this case, we may allow for bounded domains instead of the torus. However, one needs to check whether this regularization is sufficient to pass to the limit $\eps\to 0$ in the regularizing terms, as such a choice does not provide gradient bounds for $\sqrt{\rho_i}$ or $\sqrt[4]{\rho_i}$. We leave this question to future works. Third, we can derive the generalized Busenberg--Travis system \eqref{1.gBT} if the matrix $(a_{ij})$ is symmetric and positive definite. This system is fully parabolic, which simplifies the asymptotic analysis $\eps\to 0$. Indeed, we set in \eqref{1.mom} $k_i=1$ and replace the term $-\rho_i\na(\rho_1+\rho_2)$ by $\rho_i\na p_i(u)$, where $p_i(u)$ is defined in \eqref{1.gBT}. Then, using the test function $\na\log\rho_i$ in the weak formulation of \eqref{1.mom} and summing over $i=1,\ldots,n$ to derive the entropy inequality, we find that
\begin{align*}
  -\sum_{i=1}^n\int_{\T^d}\rho_i\na p_i(u)\cdot\na\log\rho_i dx
  = -\sum_{i,j=1}^n\int_{\T^d}a_{ij}\na\rho_i\cdot\na\rho_j dx
  \le -\alpha\sum_{i=1}^n\int_{\T^d}|\na\rho_i|^2 dx,
\end{align*}
where $\alpha>0$ is the smallest eigenvalue of $(a_{ij})$. This provides uniform gradient bounds for each $\rho_i$, and the Aubin--Lions compactness lemma implies the strong convergence of the approximating sequence of $\rho_i$, thus allowing us to perform the limit $\eps\to 0$.
\qed\end{remark}

The paper is organized as follows. We prove the existence of global weak solutions (Theorem \ref{thm.ex}) in Section \ref{sec.ex}. The limit $\eps\to 0$ for general $k_1$, $k_2>0$ (Theorem \ref{thm.eps}) is shown in Section \ref{sec.eps}, while Section \ref{sec.eps2} is devoted to the proof of Theorem \ref{thm.eps2} in the special case $k_1=k_2$.


\section{Proof of Theorem \ref{thm.ex}: existence of solutions}
\label{sec.ex}

We show first the existence of solutions locally in time and then derive uniform estimates from the energy \eqref{1.E} and entropy \eqref{1.H}, which allows us to extend the local solutions globally.

\subsection{Local existence of solutions}\label{sec.loc}

The local-in-time existence of solutions can be proven by the Faedo--Galerkin method; see \cite[Chap.~7]{Fei04} for the compressible Navier--Stokes equations and \cite{Jue10} for the quantum Navier--Stokes equations. Since the proof is very similar to these works, we only sketch it.

We regularize the initial data by taking $(\rho^0,u^0)\in C^\infty(\T^d;\R^4)$ such that $\rho_i^0\ge c>0$ for some $c>0$ ($i=1,2$). This is possible by using some mollifier with parameter $\delta>0$, proving the result for this initial datum and then passing to the limit $\delta\to 0$. Let $T>0$, let $(e_k)$ be an orthonormal basis of $L^2(\T^d)$ which is also an orthogonal basis of $H^1(\T^d)$, and set $X_N=\operatorname{span}\{e_1,\ldots,e_N\}$. Let the velocity $u=(u_1,u_2)\in C^0([0,T];X_N^2)$ be given and solve the approximate equations
\begin{align}\label{2.mass}
  \pa_t \rho^N_i + \diver(\rho^N_i u_i) = \delta\Delta\rho^N_i, \quad \rho^N_i(0)=\rho_i^0\quad\mbox{in }\T^d\times(0,T),\ i=1,2.
\end{align}
The maximum principle provides the lower and upper bounds $0<r_i\le\rho_i^N\le R_i$ in $\T^d\times(0,T)$, $i=1,2$, where $r_i$ and $R_i$ depend on $\delta$ and the $L^\infty(\T^d)$ norm of $\diver u_i$, and $r_i$ additionally depends on the lower bound $c>0$ of the initial data. We introduce the operator $S:C^0([0,T];X_N^2)\to C^0([0,T];C^3(\T^d;\R^2))$ by $S(u)=\rho^N=(\rho^N_1,\rho^N_2)$. This operator is Lipschitz continuous.

Next, we solve the momentum equation on the space $X_N^2$. We are looking for a solution $u^N=(u_1^N,u_2^N)\in C^0([0,T];X_N^2)$ such that for any $\psi\in C^1([0,T];X_N^2)$ with $\psi(T)=0$ and $i=1,2$,
\begin{align}
  -\eps&\int_{\T^d}\rho_i^0 u_i^0\cdot\psi(0)dx
  = \eps\int_0^T\int_{\T^d}\bigg\{\rho^N_i u_i^N\cdot\pa_t\psi
  + \rho_i^N(u_i^N\otimes u_i^N):\na\psi \nonumber \\
  &+ \rho_i^N\na\bigg(\frac{\Delta(\rho_i^N)^{1/2}}{(\rho_i^N)^{1/2}}
  \bigg)\cdot\psi - \rho_i^N\na u_i^N:\na\psi - u_i^N\cdot\psi
  - \rho_i^N|u_i^N|^2u_i^N\cdot\psi\bigg\}dxdt \label{2.inteq} \\
  &- \int_0^T\int_{\T^d}\bigg(\frac{1}{k_i}\rho_i^Nu_i^N
  + \rho_i^N\na(\rho_1^N+\rho_2^N)\bigg)\cdot\psi dxdt
  - \delta\eps\int_0^T\int_{\T^d}(u_i^N\otimes\na\rho_i^N):\na\psi dxdt.
  \nonumber
\end{align}
As already mentioned in the introduction, the additional term $\delta\eps\diver(u_i^N\otimes\na\rho_i^N)$ is introduced to deal with the term $\delta\Delta\rho_i^N$ in \eqref{2.mass} when deriving the energy estimates; see the proof of Lemma \ref{lem.EN}.
To solve problem \eqref{2.inteq}, we introduce the operator family
\begin{align*}
  \mathcal{M}[\eta]:X_N\to X_N', \quad
  \langle \mathcal{M}[\eta]u,w\rangle = \int_{\T^d}\eta u\cdot w dx,
\end{align*}
where $\eta\in L^1(\T^d)$ satisfies $\eta\ge r:=\min\{r_1,r_2\}>0$ and $u,w\in X_N$. The operator $\mathcal{M}$ is invertible and $\mathcal{M}^{-1}$ is Lipschitz continuous as a function from $L^1(\T^d)$ to the space of bounded linear mappings $X_N'\to X_N$.

We can rephrase the integral equation \eqref{2.inteq} as an ordinary differential equation on $X_N$,
\begin{align}\label{2.ode}
  \frac{d}{dt}\big(\mathcal{M}[\rho^N_i]u^N_i\big)
  = \mathcal{N}[u,u^N], \quad
  \mathcal{M}[\rho^0_i]u_i^N(0) = \mathcal{M}[\rho_i^0]u_i^0,
\end{align}
where $\rho^N=S(u)$ and
\begin{align*}
  \mathcal{N}[u,u^N] &= -\diver(\rho^N_i u\otimes u_i^N)
  + \diver(\rho_i^N\na u_i^N)
  + \rho_i^N\na\bigg(\frac{\Delta(\rho_i^N)^{1/2}}{(\rho_i^N)^{1/2}}
  \bigg) \\
  &\phantom{xx}- u_i^N
     - \rho_i^N|u_i|^2u_i^N
  - \eps^{-1}\big(k_i^{-1}\rho_i^Nu_i^N
  + \rho_i^N\na(\rho_1^N+\rho_2^N)\big)
  + \delta\diver(u_i^N\otimes\na\rho_i^N).
\end{align*}
By standard theory for systems of ordinary differential equations, for given $u\in C^0([0,T];$ $X_N^2)$, there exists a unique solution $u^N\in C^1([0,T];X_N^2)$. We are looking for a fixed point $u=u^N$ of \eqref{2.ode}, and the fixed-point equation can be written in integrated form as
\begin{align*}
  u_i^N(t) = \mathcal{M}^{-1}[S(u^N)_i(t)]
  \bigg(\mathcal{M}[\rho_i^0]u_i^0
  + \int_0^t\mathcal{N}[u^N,u^N]ds\bigg)\quad\mbox{in }X_N
\end{align*}
and using the Lipschitz continuity of $S$, $\mathcal{M}^{-1}$, and
$\mathcal{N}$, we can apply Banach's fixed-point theorem on a short time interval $[0,T^*]$ for some $0<T^*\le T$ in the space $C^0([0,T^*];X_N^2)$.

Before, we proceed with the uniform estimates, we recall the following classical lemma, which is used several times in this work.

\begin{lemma}\label{lem.weak}
Let $v$ be a weak solution to $\pa_t v + \diver F = g$ in $\T^d$, $t>0$, for some integrable functions $F$ and $g$ in the sense of
\begin{align*}
  0 = -\int_0^T\int_{\T^d}v\pa_t\chi dxdt
  - \int_{\T^d}v(0)\chi(0)dx
  - \int_0^T\int_{\T^d}(F\cdot\na\chi + g\chi)dxdt
\end{align*}
for functions $\chi\in C_0^\infty(\T^d\times[0,T))$. Then, for any $t\in[0,T]$ and $\phi\in C^\infty(\T^d\times[0,T])$,
\begin{align*}
  0 = -\int_0^t\int_{\T^d}v\pa_s\phi dxds
  + \int_{\T^d}v\phi\Big|^{s=t}_{s=0}dx
  - \int_0^t\int_{\T^d}(F\cdot\na\phi + g\phi)dxds.
\end{align*}
\end{lemma}


\subsection{Approximate energy inequality}\label{sec.ener}

To prove the global existence of solutions, it is sufficient to show that the sequence $(u^N(t))$ is bounded in $X_N^2$ for $t\in[0,T^*]$ uniformly in $T^*$.

\begin{lemma}[Energy inequality]\label{lem.EN}
Let $(\rho^N,u^N)$ be the local solution to \eqref{2.mass}--\eqref{2.inteq} constructed in Section \ref{sec.loc}. Then
\begin{align*}
  \frac{dE}{dt}&(\rho^N,u^N)
  + \sum_{i=1}^2k_i^{-1}\int_{\T^d}\rho_i^N|u_i^N|^2 dx
  + \eps\sum_{i=1}^2\int_{\T^d}\big(|u_i^N|^2 + \rho_i^N|\na u_i^N|^2
  + \rho_i^N|u_i^N|^4\big)dx \\
  &+ \delta\int_{\T^d}| \na(\rho_1^N+\rho_2^N)|^2dx
  + \frac{\delta\eps}{2}\sum_{i=1}^2\int_{\T^d}
  \rho_i^N|D^2\log\rho_i^N|^2 dx = 0,
\end{align*}
recalling Definition \eqref{1.E} of $E(\rho^N,u^N)$.
\end{lemma}

Since the lower bound $\rho_i^N\ge r_i>0$ yields a uniform estimate for $(u_i^N)$ in $L^2(\T^d)$ uniformly in time, thanks to the definition of the energy \eqref{1.E}, we obtain the desired estimate for $u^N$.

\begin{proof}
We use the test function $\phi= \eps|u_i^N|^2/2-\eps\Delta(\rho_i^N)^{1/2}/(\rho_i^N)^{1/2}+(\rho_1^N+\rho_2^N)$ in the weak formulation of \eqref{2.mass} and the test function $\psi=u_i^N$ in the time-differentiated form of \eqref{2.inteq} (see Lemma \ref{lem.weak}) and add both equations. Then, using the identity $\pa_t(\rho_i^N u_i^N)+\diver(\rho_i^N u_i^N\otimes u_i^N) = \rho_i^N(\pa_t u_i^N+u_i^N\cdot\na u_i^N) + \delta \Delta \rho_i^N u_i^N$ and proceeding as in the proof of \cite[Lemma 3.1]{Jue10}, some terms cancel, and we end up with
\begin{align*}
  \eps\frac{d}{dt}&\int_{\T^d}\big(\rho_i^N|u_i^N|^2
  + |\na(\rho_i^N)^{1/2}|^2\big)dx
  + \int_{\T^d}\pa_t\rho_i^N(\rho_1^N+\rho_2^N)dx \\
  &= -\int_{\T^d}\big(k_i^{-1}\rho_i^N|u_i^N|^2
  + \eps(|u_i^N|^2 + \rho_i^N|\na u_i^N|^2
  + \rho_i^N|u_i^N|^4)\big)dx \\
  &\phantom{xx}- \delta\eps\int_{\T^d}\Delta\rho_i^N
  \frac{\Delta(\rho_i^N)^{1/2}}{(\rho_i^N)^{1/2}}dx
  - \delta\int_{\T^d}\na\rho_i^N\cdot\na(\rho_1^N+\rho_2^N)dx.
\end{align*}
Adding these equations for $i=1,2$ and using
\begin{align}\label{2.second}
  \int_{\T^d}\Delta\rho_i^N\frac{\Delta(\rho_i^N)^{1/2}}{(\rho_i^N)^{1/2}}
  dx = \frac12\int_{\T^d}\rho_i^N|D^2\log \rho_i^N|^2 dx
\end{align}
(see \cite[(3.7)]{Jue10}), we find that
\begin{align*}
  \frac{d}{dt}&\int_{\T^d}\bigg(\eps\sum_{i=1}^2\big(\rho_i^N|u_i^N|^2
  + |\na(\rho_i^N)^{1/2}|^2\big) + \frac12(\rho_1+\rho_2)^2\bigg)dx \\
  &+ \sum_{i=1}^2\int_{\T^d}\big(k_i^{-1}\rho_i^N|u_i^N|^2
  + \eps(|u_i^N|^2 + \rho_i^N|\na u_i^N|^2
  + \rho_i^N|u_i^N|^4)\big)dx \\
  &+ \delta\int_{\T^d}|\na(\rho_1^N+\rho_2^N)|^2dx
  + \frac{\delta\eps}{2}\sum_{i=1}^2\int_{\T^d}\rho_i^N
  |D^2\log \rho_i^N|^2 dx = 0.
\end{align*}
This finishes the proof.
\end{proof}

The energy inequality and inequality \eqref{1.ineq} imply the following bounds.

\begin{corollary}[Uniform estimates I]\label{lem.est1}
Let $(\rho^N,u^N)$ be the local solution to \eqref{2.mass}--\eqref{2.inteq} constructed in Section \ref{sec.loc}. Then there exists $C>0$ independent of $(\delta,\eps,N)$ such that for $i=1,2$,
\begin{align*}
  \|\rho_i^N\|_{L^\infty(0,\infty;L^2(\T^d))}
  + \|(\rho_i^N)^{1/2}u_i^N\|_{L^2(0,\infty;L^2(\T^d))}
  + \sqrt{\eps}\|\na(\rho_i^N)^{1/2}\|_{L^\infty(0,\infty;L^2(\T^d))}
  &\le C, \\
  \sqrt{\eps}\|(\rho_i^N)^{1/2}u_i^N\|_{L^\infty(0,\infty;L^2(\T^d))}
  + \sqrt{\eps}\|(\rho_i^N)^{1/2}\na u_i^N\|_{L^2(0,\infty;L^2(\T^d))}
  &\le C, \\
  \sqrt{\eps}\|u_i^N\|_{L^2(0,\infty;L^2(\T^d))}
  + \sqrt[4]{\eps}\|(\rho_i^N)^{1/4}u_i^N\|_{L^4(0,\infty;L^4(\T^d))}
  &\le C, \\
  \sqrt{\delta\eps}\|(\rho_i^N)^{1/2}D^2\log\rho_i^N\|_{L^2(0,\infty;
  L^2(\T^d))} + \sqrt[4]{\delta\eps}
  \|\na(\rho_i^N)^{1/4}\|_{L^4(0,\infty;L^4(\T^d))}  &\le C.
\end{align*}
\end{corollary}


\subsection{Approximate entropy inequality}\label{sec.ent}

Further uniform estimates are derived from the entropy inequality. Here, we work directly with the weak formulation of \eqref{2.mass} and with the weak formulation \eqref{2.inteq}.

\begin{lemma}[Entropy inequality]\label{lem.HN}
Let $(\rho^N,u^N)$ be the local solution to \eqref{2.mass}--\eqref{2.inteq} constructed in Section \ref{sec.loc}. Then
\begin{align*}
  H(\rho^N&(t)) + \int_0^t\int_{\T^d}
  |\na(\rho_1^N+\rho_2^N)|^2 dxds
  + \frac{\eps}{8}\sum_{i=1}^2\int_0^t\int_{\T^d}
  \rho_i^N|D^2\log\rho_i^N|^2 dxds \\
  &+ \sum_{i=1}^2\frac{4\delta}{k_i}\int_0^t\int_{\T^d}
  |\na(\rho_i^N)^{1/2}|^2dxds
  + \delta\eps\sum_{i=1}^2\int_0^t\int_{\T^d}|\na\log\rho_i^N|^2dxds
  \le C,
\end{align*}
and the constant $C>0$ only depends on the initial data.
\end{lemma}

\begin{proof}
Recalling that $\rho_i^N$ is smooth and positive, we compute
\begin{align}\label{2.dHdt}
  H(\rho^N&(t)) - H(\rho^N(0))
  = \sum_{i=1}^2\frac{1}{k_i}\int_0^t\int_{\T^d}\pa_s\rho_i^N
  \log\rho_i^N dxds \\
  &= \sum_{i=1}^2\frac{1}{k_i}\int_0^t\int_{\T^d}\big(\rho_i^N u_i^N
  - \delta\na\rho_i^N\big)\cdot\na\log\rho_i^N dxds \nonumber \\
  &= \sum_{i=1}^2\frac{1}{k_i}\int_0^t\int_{\T^d}u_i^N
  \cdot\na\rho_i^N dxds
  - \sum_{i=1}^2\frac{4\delta}{k_i}\int_0^t\int_{\T^d}
  |\na(\rho_i^N)^{1/2}|^2 dxds. \nonumber
\end{align}
To estimate the first term on the right-hand side, we use the test function $\na\log\rho_i^N$ in \eqref{2.inteq}:
\begin{align}\label{2.I17}
  & \frac{1}{k_i}\int_0^t\int_{\T^d}u_i^N\cdot\na\rho_i^N dxds
  = \frac{1}{k_i}\int_0^t\int_{\T^d}(\rho_i^Nu_i^N)
  \cdot\na\log\rho_i^N dxds = I_1+\cdots+I_7, \\
  & I_1 = -\eps\int_0^t\int_{\T^d}\big(\pa_s(\rho_i^N u_i^N)
  + \diver(\rho_i^N u_i^N\otimes u_i^N)\big)\cdot\na\log\rho_i^N dxds,
  \nonumber \\
  & I_2 = \eps\int_0^t\int_{\T^d}\rho_i^N\na
  \frac{\Delta(\rho_i^N)^{1/2}}{(\rho_i^N)^{1/2}}
  \cdot\na\log\rho_i^N dxds, \nonumber \\
  & I_3 = \eps\int_0^t\int_{\T^d}\diver(\rho_i^N\na u_i^N)
  \cdot\na\log\rho_i^N dxds, \nonumber \\
  & I_4 = -\eps\int_0^t\int_{\T^d}u_i^N\cdot\na\log\rho_i^N dxds,
  \nonumber \\
  & I_5 = -\eps\int_0^t\int_{\T^d}\rho_i^N|u_i^N|^2u_i^N
  \cdot\na\log\rho_i^N dxds,  \nonumber \\
  & I_6 = -\int_0^t\int_{\T^d}\rho_i^N\na(\rho_1^N+\rho_2^N)
  \cdot\na\log\rho_i^N dxds, \nonumber \\
  & I_7 = \delta\eps\int_0^t\int_{\T^d}\diver(u_i^N\otimes\na\rho_i^N)
  \cdot\na\log\rho_i^N dxds. \nonumber
\end{align}

{\em Step 1: Estimation of $I_1$, $I_2$, and $I_7$:} We start with $I_2$. We infer from identity \eqref{2.second} that
\begin{align*}
  I_2 = -\eps\int_0^t\int_{\T^d}\frac{\Delta(\rho_i^N)^{1/2}}{
  (\rho_i^N)^{1/2}}\Delta\rho_i^N dxds
  = -\frac{\eps}{2}\int_0^t\int_{\T^d}\rho_i^N|D^2\log\rho_i^N|^2 dxds,
\end{align*}
and this expression will be used to absorb some integrals coming from the other terms. It follows from \eqref{2.mass} that
\begin{align}\label{2.I1}
  I_1 &= -\eps\int_0^t\int_{\T^d}\big(\rho_i^N(\pa_s u_i^N
  + u_i^N\cdot\na u_i^N) + \delta\Delta\rho_i^N u_i^N\big)
  \cdot\na\log\rho_i^N dxds \\
  &= -\eps\int_0^t\int_{\T^d}\big(\pa_s u_i^N\cdot\na\rho_i^N
  + u_i^N\cdot\na u_i^N\cdot\na\rho_i^N
  + \delta\Delta\rho_i^N u_i^N\cdot\na\log\rho_i^N\big)dxds \nonumber \\
  &=: I_{11}+I_{12}+I_{13}. \nonumber
\end{align}
The term $I_{11}$ is rewritten according to
\begin{align*}
  I_{11} &= -\eps\int_0^t\int_{\T^d}\big(\pa_s(u_i^N\cdot\na\rho_i^N)
  - u_i^N\cdot\na\pa_s\rho_i^N\big) dxds \\
  &= -\eps\int_0^t\frac{d}{ds}
      \int_{\T^d}u_i^N\cdot\na\rho_i^N dxds
  - \eps\int_0^t\int_{\T^d}\diver u_i^N\big(-\diver(\rho_i^N u_i^N)
  + \delta\Delta\rho_i^N\big) dxds \\
  &= -\eps\int_{\T^d}u_i^N(t)\cdot\na\rho_i^N(t) dx
  + \eps\int_{\T^d}u_i^N(0)\cdot\na\rho_i^N(0)dx \\
  &\phantom{xx}+ \eps\int_0^t\int_{\T^d}
  (u_i^N\cdot\na\rho_i^N)\diver u_i^N dxds
  + \eps\int_0^t\int_{\T^d}\rho_i^N(\diver u_i^N)^2 dxds \\
  &\phantom{xx}
  - \delta\eps\int_0^t\int_{\T^d}\Delta\rho_i^N \diver u_i^N dxds
  =: I_{111}+\cdots+I_{115}.
\end{align*}
Corollary \ref{lem.est1} shows that, for $0<t<T$,
\begin{align*}
  I_{111} + I_{112} &= -\eps\int_{\T^d}u_i^N(t)\cdot\na\rho_i^N(t)dx
  + \eps\int_{\T^d}u_i^N(0)\cdot\na\rho_i^N(0)dx \\
  &\le 2\big(\sqrt{\eps}\|(\rho_i^N)^{1/2}u_i^N
  \|_{L^\infty(0,T;L^2(\Omega))}\big)\big(\sqrt{\eps}
  \|\na(\rho_i^N)^{1/2}\|_{L^\infty(0,T;L^2(\T^d))}\big) \\
  &\phantom{xx} + 2\eps\|(\rho_i^0)^{1/2}u_i^0\|_{L^2(\T^d)}
  \|\na(\rho_i^0)^{1/2}\|_{L^2(\T^d)} \le C, \\
  I_{114} &\le \eps\|(\rho_i^N)^{1/2}\na u_i^N\|_{L^2(0,T;L^2(\T^d))}^2
  \le C.
\end{align*}
Furthermore, using H\"older's inequality, Corollary \ref{lem.est1}, and inequality \eqref{1.ineq},
\begin{align*}
  I_{113} &\le 4\big(\sqrt[4]{\eps}\|(\rho_i^N)^{1/4}u_i^N
  \|_{L^4(0,T;L^4(\T^d))}\big)\big(\sqrt{\eps}\|(\rho_i^N)^{1/2}
  \na u_i^N\|_{L^2(0,T;L^2(\T^d))}\big) \\
  &\phantom{xx}\times\big(\sqrt[4]{\eps}
  \|\na(\rho_i^N)^{1/4}\|_{L^4(0,T;L^4(\T^d))}\big) \\
  &\le C + \frac{\eps}{16c(d)}
  \int_0^t\int_{\T^d}|\na(\rho_i^N)^{1/4}|^4 dxds \\
  &\le C + \frac{\eps}{16}\int_0^t\int_{\T^d}
  \rho_i^N|D^2\log\rho_i^N|^2 dxds,
\end{align*}
where in the last but one step we applied Young's inequality. For the remaining term $I_{115}$, we first reformulate it, and then use similar arguments as for $I_{113}$:
\begin{align*}
  I_{115} &= -\delta\eps\int_0^t\int_{\T^d}
  \diver(\rho_i^N\na\log\rho_i^N)\diver u_i^N dxds \\
  &= -\delta\eps\int_0^t\int_{\T^d}\big(\rho_i^N\Delta\log\rho_i^N
  + \na\rho_i^N\cdot\na\log\rho_i^N\big)\diver u_i^N dxds \\
  &= -\delta\eps\int_0^t\int_{\T^d}\big((\rho_i^N)^{1/2}
  \Delta\log\rho_i^N + 16|\na(\rho_i^N)^{1/4}|^2\big)(\rho_i^N)^{1/2}
  \diver u_i^N dxds \\
  &\le \delta\sqrt{\eps}\big(\|(\rho_i^N)^{1/2}\Delta\log\rho_i^N
  \|_{L^2(0,T;L^2(\T^d))}
  + 16\|\na(\rho_i^N)^{1/4}\|_{L^4(0,T;L^4(\T^d))}^2\big) \\
  &\phantom{xx}\times\sqrt{\eps}
  \|(\rho_i^N)^{1/2}\na u_i^N\|_{L^2(0,T;L^2(\T^d))} \\
  &\le C + \frac{\eps}{16}\int_0^t\int_{\T^d}
  \rho_i^N|D^2\log\rho_i^N|^2 dxds.
\end{align*}

In a similar way, the bounds in Corollary \ref{lem.est1} yield
\begin{align*}
  I_{12} &\le 4\big(\sqrt[4]{\eps}\|(\rho_i^N)^{1/4}u_i^N
  \|_{L^4(0,T;L^4(\T^d))}\big)\big(\sqrt{\eps}\|(\rho_i^N)^{1/2}
  \na u_i^N\|_{L^2(0,T;L^2(\T^d))}\big) \\
  &\phantom{xx}\times\big(\sqrt[4]{\eps}\|\na(\rho_i^N)^{1/4}
  \|_{L^4(0,T;L^4(\T^d))}\big) \\
  &\le C + \frac{\eps}{16}\int_0^t\int_{\T^d}
  \rho_i^N|D^2\log\rho_i^N|^2 dxds.
\end{align*}
We conclude from \eqref{2.I1} that
\begin{align*}
  I_1 \le C + \frac{3\eps}{16}\int_0^t\int_{\T^d}
  \rho_i^N|D^2\log\rho_i^N|^2 dxds - \delta\eps\int_0^t\int_{\T^d}
  \Delta\rho_i^N u_i^N\cdot\na\log\rho_i^N dxds,
\end{align*}
and the last term cancels with a part of $I_7$, since
\begin{align*}
  I_7 &= \delta\eps\int_0^t\int_{\T^d}\big(\na\log\rho_i^N
  \cdot\na u_i^N\cdot\na\rho_i^N
  + \Delta\rho_i^Nu_i^N\cdot\na\log\rho_i^N\big)dxds \\
  &\le \sqrt{\delta}\big(4\sqrt[4]{\delta\eps}\|\na(\rho_i^N)^{1/4}
  \|_{L^4(0,T;L^4(\T^d))}\big)^2\big(
  \sqrt{\eps}\|(\rho_i^N)^{1/2}\na u_i^N \|_{L^2(0,T;L^2(\T^d))}\big) \\
  &\phantom{xx}+ \delta\eps\int_0^t\int_{\T^d}
  \Delta\rho_i^Nu_i^N\cdot\na\log\rho_i dxds \\
  &\le C + \delta\eps\int_0^t\int_{\T^d}
  \Delta\rho_i^Nu_i^N\cdot\na\log\rho_i dxds.
\end{align*}
This shows that
\begin{align}\label{2.I127}
  I_1 + I_2 + I_7 \le C - \frac{5\eps}{16}
  \int_0^t\int_{\T^d}\rho_i^N|D^2\log\rho_i^N|^2 dxds.
\end{align}

{\em Step 2: Estimation of $I_3,\ldots,I_6$:}
We continue with the estimate of $I_3$:
\begin{align*}
  I_3 &= -\eps\int_0^t\int_{\T^d}
  \rho_i^N\na u_i^N:D^2\log\rho_i^N dxds \\
  &\le \frac{\eps}{16}\int_0^t\int_{\T^d}\rho_i^N|D^2\log\rho_i^N|^2 dxds
  + 4\eps\int_0^t\int_{\T^d}\rho_i^N|\na u_i^N|^2 dxds \\
  &\le \frac{\eps}{16}\int_0^t\int_{\T^d}\rho_i^N|D^2\log\rho_i^N|^2 dxds
  + C.
\end{align*}
By the approximative mass balance equation \eqref{2.mass}, we have
\begin{align*}
  I_4 &= -\eps\int_0^t\int_{\T^d}\frac{u_i^N\cdot\na\rho_i^N}{\rho_i^N}
  dxds = \eps\int_0^t\int_{\T^d}\frac{1}{\rho_i^N}\big(\pa_s\rho_i^N
  + \rho_i^N\diver u_i^N - \delta\Delta\rho_i^N\big)dxds \\
  &= \eps\int_0^t\int_{\T^d}\pa_s\log\rho_i^N dxds
  + \eps\int_0^t\int_{\T^d}\diver u_i^N dxds
  - \delta\eps\int_0^t\int_{\T^d}|\na\log\rho_i^N|^2 dxds \\
  &= \eps\int_{\T^d}\log\rho_i^N(t)dx - \eps\int_{\T^d}\log\rho_i^N(0)dx
  - \delta\eps\int_0^t\int_{\T^d}|\na\log\rho_i^N|^2 dxds \\
  &\le \eps\int_{\T^d}(\rho_i^N(t)-1)dx
  - \eps\int_{\T^d}\log\rho_i^0 dx
  - \delta\eps\int_{\T^d}|\na\log\rho_i^N|^2 dxds.
\end{align*}
Because of the $L^2(\T^d)$ bound for $\rho_i^N$ in Corollary \ref{lem.est1}, the first term on the right-hand side is bounded uniformly in $(N,\delta,\eps)$. The same holds true for the second term on the right-hand side, since $\log\rho_i^0$ is assumed to be integrable. We infer that
\begin{align*}
  I_4 \le C - \delta\eps\int_{\T^d}|\na\log\rho_i^N|^2 dxds.
\end{align*}
Furthermore, using H\"older's inequality, Young's inequality, and then inequality \eqref{1.ineq},
\begin{align*}
  I_5 &\le 4\big(\sqrt[4]{\eps}\|(\rho_i^N)^{1/4}u_i^N
  \|_{L^{4}(0,\infty;L^{4}(\T^d))}\big)^3
  \big(\sqrt[4]{\eps}\|\na(\rho_i^N)^{1/4}
  \|_{L^4(0,\infty;L^4(\T^d))}\big) \\
  &\le C + \frac{\eps}{16}\int_0^t\int_{\T^d}\rho_i^N|D^2\log\rho_i^N|^2
  dxds.
\end{align*}
Finally, the term $I_6$ is rewritten as
\begin{align*}
  I_6 = -\int_0^t\int_{\T^d}\na\rho_i^N\cdot\na(\rho_1^N+\rho_2^N)dxds,
\end{align*}
and it becomes nonpositive when added for $i=1,2$. We conclude that
\begin{align}\label{2.I3456}
  I_3+\cdots+I_6 &\le C + \frac{\eps}{8}\int_0^t\int_{\T^d}
  \rho_i^N|D^2\log\rho_i^N|^2 dxds
  - \delta\eps\int_{\T^d}|\na\log\rho_i^N|^2 dxds \\
  &\phantom{xx}
  -\int_0^t\int_{\T^d}\na\rho_i^N\cdot\na(\rho_1^N+\rho_2^N)dxds.
  \nonumber
\end{align}

{\em Step 3: End of the proof:} We insert \eqref{2.I127} and \eqref{2.I3456} into \eqref{2.I17} and sum over $i=1,2$ to conclude from \eqref{2.dHdt} the desired entropy inequality.
\end{proof}

The entropy inequality allows us to improve some bounds from Corollary \ref{lem.est1}.

\begin{corollary}[Uniform estimates II]\label{lem.est2}
Let $(\rho^N,u^N)$ be the local solution to \eqref{2.mass}--\eqref{2.inteq} constructed in Section \ref{sec.loc}. Then there exists $C>0$ independent of $(\delta,\eps,N)$ such that
\begin{align*}
  \|\na(\rho_1^N+\rho_2^N)\|_{L^2(0,\infty;L^2(\T^d))} &\le C, \\
  \sqrt{\eps}\|(\rho_i^N)^{1/2}\|_{L^2(0,\infty;H^2(\T^d))}
  + \sqrt[4]{\eps}\|\na(\rho_i^N)^{1/4}\|_{L^4(0,\infty;L^4(\T^d))}
  &\le C.
\end{align*}
\end{corollary}


\subsection{Further uniform estimates}

We derive some spatial and time regularity bounds for $\rho_i^N$ and $\rho_i^N u_i^N$ uniform in $N$, which are needed for the limit $N\to\infty$.

\begin{lemma}[Spatial regularity]\label{lem.space}
For any $T>0$, there exists a constant $C(\eps)>0$ independent of $N$ and $\delta$ such that
\begin{align*}
  \|\rho_i^N\|_{L^\infty(0,T;L^3(\T^d))}
  + \|(\rho_i^N)^{1/2}\|_{L^4(0,T;W^{1,3}(\T^d))} &\le C(\eps), \\
  \|\rho_i^N\|_{L^2(0,T;W^{2,3/2}(\T^d))}
  + \|\rho_i^N u_i^N\|_{L^2(0,T;W^{1,3/2}(\T^d))} &\le C(\eps).
\end{align*}
\end{lemma}

\begin{proof}
The first bound for $\rho_i^N$ is an immediate consequence of Corollary \ref{lem.est1} and the Sobolev embedding $H^1(\T^d)\hookrightarrow L^6(\T^d)$, yielding a uniform bound for $(\rho_i^N)^{1/2}$ in $L^\infty(0,T;L^6(\T^d))$ (uniform in $N$ and $\delta$). It follows from the embedding $H^2(\T^d)\hookrightarrow L^\infty(\T^d)$ (here we use the condition $d\le 3$) and Corollary \ref{lem.est2} that $(\rho_i^N)^{1/2}$ is uniformly bounded in $L^2(0,T;L^\infty(\T^d))$. Then, together with the uniform bound for $(\rho_i^N)^{1/2}u_i^N$ in $L^\infty(0,T;L^2(\T^d))$ from Corollary \ref{lem.est1}, we obtain a uniform estimate for $\rho_i^Nu_i^N=(\rho_i^N)^{1/2}\cdot(\rho_i^N)^{1/2}u_i^N$ in $L^2(0,T;L^2(\T^d))$. Furthermore, by Corollary \ref{lem.est2}, $(\na(\rho_i^N)^{1/2})$ is bounded in $L^2(0,T;L^6(\T^d))$. Then
\begin{align*}
  \na(\rho_i^N u_i^N) = 2\na(\rho_i^N)^{1/2}\cdot(\rho_i^N)^{1/2}u_i^N
  + 2(\rho_i^N)^{1/2}\cdot(\rho_i^N)^{1/2}\na u_i^N
\end{align*}
is uniformly bounded in $L^2(0,T;L^{3/2}(\T^d))$. This yields a uniform bound for $\rho_i^Nu_i^N$ in $L^2(0,T;W^{1,3/2}(\T^d))$.

Next, we apply the Gagliardo--Nirenberg inequality with $\theta=1/2$:
\begin{align*}
  \|\na(\rho_i^N)^{1/2}\|_{L^4(0,T;L^3(\T^d))}^4
  &\le C\int_0^T\|(\rho_i^N)^{1/2}\|_{H^2(\T^d)}^{4\theta}
  \|(\rho_i^N)^{1/2}\|_{L^6(\T^d)}^{4(1-\theta)}dt \\
  &\le C\|(\rho_i^N)^{1/2}\|_{L^\infty(0,T;L^6(\T^d))}^{2}
  \int_0^T\|(\rho_i^N)^{1/2}\|_{H^2(\T^d)}^2 dt \le C,
\end{align*}
showing that $(\rho_i^N)^{1/2}$ is uniformly bounded in $L^4(0,T;W^{1,3}(\T^d))$. Because of inequality \eqref{1.ineq}, the uniform bound for $(\rho_i^N)^{1/2}D^2\log\rho_i^N$ in $L^2(0,T;L^2(\T^d))$ implies a
uniform bound for $D^2(\rho_i^N)^{1/2}$ in $L^2(0,T;$ $L^2(\T^d))$. Thus,
\begin{align*}
  D^2\rho_i^N = 2(\rho_i^N)^{1/2}D^2(\rho_i^N)^{1/2}
  + 2\na(\rho_i^N)^{1/2}\otimes\na(\rho_i^N)^{1/2}
\end{align*}
is uniformly bounded in $L^2(0,T;L^{3/2}(\T^d))$. Finally, the bound for $(\rho_i^N)^{1/4}$ in $L^\infty(0,T;$ $L^{12}(\T^d))$ and for $\na(\rho_i^N)^{1/4}$ in $L^4(0,T;L^4(\T^d))$ yield an estimate for $\na(\rho_i^N)^{1/2}=2(\rho_i^N)^{1/4}$ $\na(\rho_i^N)^{1/4}$ in $L^4(0,T;L^3(\T^d))$, finishing the proof.
\end{proof}

\begin{lemma}[Time regularity]
For any $T>0$, there exists a constant $C(\eps)>0$ independent of $N$ and $\delta$ such that for $s>d/2+1$,
\begin{align*}
  \|\pa_t\rho_i^N\|_{L^2(0,T;L^{3/2}(\T^d))}
  + \|\pa_t(\rho_i^Nu_i^N)\|_{L^{4/3}(0,T;H^{s}(\T^d)')}
  + \|\pa_t(\rho_i^N)^{1/2}\|_{L^2(0,T;L^{2}(\T^d))} \le C.
\end{align*}
\end{lemma}

\begin{proof}
We deduce from Lemma \ref{lem.space} that
\begin{align*}
  \pa_t\rho_i^N = -\diver(\rho_i^Nu_i^N) + \delta\Delta\rho_i^N
\end{align*}
is uniformly bounded in $L^2(0,T;L^{3/2}(\T^d))$. The estimate on $\pa_t(\rho_i^Nu_i^N)$ follows from the following bounds, which are consequences of Corollaries \ref{lem.est1} and \ref{lem.est2}, as well as from the spatial regularity of Lemma \ref{lem.space}:
\begin{itemize}
\item The sequence $(\rho_i^N u_i^N\otimes u_i^N)$ is bounded in $L^\infty(0,T;L^1(\T^d))$. Hence, $\diver(\rho_i^N u_i^N\otimes u_i^N)$ is bounded in $L^\infty(0,T;H^s(\T^d)')$, since $H^s(\T^d)$ $\hookrightarrow W^{1,\infty}(\T^d)$ for $s>d/2+1$.
\item We know that $\rho_i^N\na u_i^N
=(\rho_i^N)^{1/2}\cdot(\rho_i^N)^{1/2}\na u_i^N$ is bounded in $L^2(0,T;L^{3/2}(\T^d))$. Thus, $\diver(\rho_i^N\na u_i^N)$ is bounded in $L^2(0,T;W^{1,3}(\T^d)')\hookrightarrow L^2(0,T;H^s(\T^d)')$.
\item Let $\psi\in L^4(0,T;W^{1,3}(\T^d;\R^d))$. Then, by integration by parts,
\begin{align*}
  \int_0^T&\int_{\T^d}\rho_i^N\na\bigg(\frac{\Delta(\rho_i^N)^{1/2}}{
  (\rho_i^N)^{1/2}}\bigg)\cdot\psi dxds \\
  &= -\int_0^T\int_{\T^d}\Delta(\rho_i^N)^{1/2}\big(2\na(\rho_i^N)^{1/2}
  \cdot\psi + (\rho_i^N)^{1/2}\diver\psi\big)dxds \\
  &\le 2\|\Delta(\rho_i^N)^{1/2}\|_{L^2(0,T;L^2(\T^d))}
  \|\na(\rho_i^N)^{1/2}\|_{L^4(0,T;L^{3}(\T^d))}
  \|\psi\|_{L^4(0,T;L^6(\T^d))} \\
  &\phantom{xx} + \|\Delta(\rho_i^N)^{1/2}\|_{L^2(0,T;L^2(\T^d))}
  \|(\rho_i^N)^{1/2}\|_{L^\infty(0,T;L^6(\T^d))}
  \|\psi\|_{L^2(0,T;W^{1,3}(\T^d))} \\
  &\le C\|\psi\|_{L^4(0,T;W^{1,3}(\T^d))}.
\end{align*}
This proves that $\rho_i^N\na(\Delta(\rho_i^N)^{1/2}/(\rho_i^N)^{1/2})$ is uniformly bounded in $L^{4/3}(0,T;W^{1,3}(\T^d)')$ $\hookrightarrow L^{4/3}(0,T;H^s(\T^d)')$.
\item The sequence $\rho_i^N|u_i^N|^2 u_i^N = (\rho_i^N)^{1/2}u_i^N\cdot
(\rho_i^N)^{1/2}|u_i^N|^2$ is bounded in $L^2(0,T;L^1(\T^d))$ $\hookrightarrow L^{2}(0,T;H^s(\T^d)')$, since $(\rho_i^N)^{1/2}u_i^N$ is bounded in $L^\infty(0,T;L^2(\T^d))$ and $(\rho_i^N)^{1/2}$ $|u_i^N|^2$ is bounded in $L^2(0,T;L^2(\T^d))$ by Corollary \ref{lem.est1}.
\item By Lemma \ref{lem.space}, $\rho_i^N u_i^N$ is uniformly bounded in $L^2(0,T;W^{1,3/2}(\T^d))\hookrightarrow L^2(0,T;$ $L^{3/2}(\T^d))$.
\item In view of the $L^\infty(0,T;L^3(\T^d))$ bound of $\rho_i^N$ from Lemma \ref{lem.space} and the $L^2(0,T;L^2(\T^d))$ bound of $\na(\rho_1^N+\rho_2^N)$ from Corollary \ref{lem.est2}, $\rho_i^N\na(\rho_1^N+\rho_2^N)$ is bounded in $L^2(0,T;L^{6/5}(\T^d))$ $\hookrightarrow L^2(0,T;H^s(\T^d)')$.
\item We deduce from the $L^2(0,T;L^2(\T^d))$ bound of $(\rho_i^N)^{1/2}u_i^N$ (Corollary \ref{lem.est1}) and the $L^4(0,T;L^3(\T^d))$ bound of $\na(\rho_i^N)^{1/2}$ that
$u_i^N\otimes\na\rho_i^N=2(\rho_i^N)^{1/2}u_i^N\otimes\na(\rho_i^N)^{1/2}$ is bounded in $L^{4/3}(0,T;L^{6/5}(\T^d))$. Hence, the sequence $\diver(u_i^N\otimes\na\rho_i^N)$ is bounded in $L^{4/3}(0,T;W^{1,6}(\T^d)')\hookrightarrow L^{4/3}(0,T;H^s(\T^d)')$.
\end{itemize}
We conclude that
\begin{align*}
  \pa_t(\rho_i^Nu_i^N) &= -\diver(\rho_i^Nu_i^N\otimes u_i^N)
  + \rho_i^N\na\bigg(\frac{\Delta(\rho_i^N)^{1/2}}{ (\rho_i^N)^{1/2}}
  \bigg) + \diver(\rho_i^N\na u_i^N) - u_i^N \\
  &\phantom{xx}- \rho_i^N u_i^N|u_i^N|^2
  - \eps^{-1}k_i^{-1}\rho_i^Nu_i^N - \eps^{-1}\rho_i^N\na(\rho_1+\rho_2^N)
  + \delta\diver(u_i^N\otimes\na\rho_i^N)
\end{align*}
is uniformly bounded in $L^{4/3}(0,T;H^s(\T^d)')$. Finally, the sequence
\begin{align*}
  \pa_t(\rho_i^N)^{1/2} &= -\frac12(\rho_i^N)^{1/2}\diver u_i^N
  - 2\na(\rho_i^N)^{1/4}\cdot((\rho_i^N)^{1/4}u_i^N) \\
  &\phantom{xx}+ \delta\Delta(\rho_i^N)^{1/2}
  + 4\delta|\na(\rho_i^N)^{1/4}|^2
\end{align*}
is bounded in $L^2(0,T;L^2(\T^d))$.
\end{proof}


\subsection{Limit $N\to\infty$}

The spatial and time regularity for $(\rho_i^N)^{1/2}$, $\rho_i^N$, and $\rho_i^Nu_i^N$ allow us to apply the Aubin--Lions compactness lemma to conclude the existence of a subsequence (not relabeled) such that, as $N\to\infty$,
\begin{align*}
  (\rho_i^N)^{1/2}\to\sqrt{\rho_i} &\quad\mbox{strongly in }
  L^2(0,T;H^1(\T^d)), \\
  \rho_i^N\to\rho_i &\quad\mbox{strongly in }L^2(0,T;L^2(\T^d)), \\
  \rho_i^Nu_i^N \to J_i &\quad\mbox{strongly in }L^2(0,T;L^2(\T^d)).
\end{align*}
Furthermore, we have the weak convergences (up to subsequences)
\begin{align*}
  (\rho_i^N)^{1/2}\rightharpoonup \sqrt{\rho_i} &\quad\mbox{weakly in }
  L^2(0,T;H^2(\T^d)), \\
  u_i^N\rightharpoonup u_i &\quad\mbox{weakly in }L^2(0,T;L^2(\T^d)).
\end{align*}
It follows that $\rho_i^N u_i^N\rightharpoonup \rho_iu_i$ weakly in $L^1(0,T;L^1(\T^d))$, showing that $J_i=\rho_iu_i$. Moreover, by Lemma \ref{lem.space}, $\delta\na\rho_i^N=2\delta(\rho_i^N)^{1/2}\na(\rho_i^N)^{1/2}\to 0$ strongly in $L^4(0,T;L^2(\T^d))$.

With these convergences, we can pass to the limit $N\to\infty$ and $\delta\to 0$ in \eqref{2.mass}, formulated for $\phi\in C_0^\infty(\T^d\times[0,T))$ as
\begin{align*}
  0 &= -\int_0^T\int_{\T^d}\rho_i^N\pa_t\phi dxdt
  - \int_{\T^d}\rho_i^0\phi(0)dx
  - \int_0^T\int_{\T^d}\rho_i^N u_i^N\cdot\na\phi dxdt \\
   &\phantom{xx}+ \delta\int_0^T\int_{\T^d}\na\rho_i^N\cdot\na\phi dxdt,
\end{align*}
leading to
\begin{align*}
  0 = -\int_0^T\int_{\T^d}\rho_i\pa_t\phi dxdt
  - \int_{\T^d}\rho_i^0\phi(0)dx
  - \int_0^T\int_{\T^d}\rho_i u_i\cdot\na\phi dxdt.
\end{align*}

The limit in the momentum balance equation \eqref{2.inteq} is more involved. The strong convergence of $\rho_i^Nu_i^N$ and the weak convergence of $u_i^N$ in $L^2(0,T;L^2(\T^d))$ lead to
\begin{align*}
  \rho_i^Nu_i^N\otimes u_i^N\rightharpoonup \rho_iu_i\otimes u_i
  \quad\mbox{weakly in }L^1(0,T;L^1(\T^d)).
\end{align*}
Similarly, $(\rho_i^N)^{1/2}u_i^N\rightharpoonup\sqrt{\rho_i}u_i$ weakly in $L^2(0,T;L^2(\T^d))$, taking into account Corollary \ref{lem.est1}. Then, together with the strong convergences of $\na(\rho_i^N)^{1/2}$ and $\rho_i^Nu_i^N$ in $L^2(0,T;L^2(\T^d))$, we find that for $\psi\in C_0^\infty(\T^d\times(0,T);\R^{d\times d}))$, integrating by parts,
\begin{align}\label{rhonau}
  -\int_0^T&\int_{\T^d}\rho_i^N\na u_i^N:\psi dxdt
  = \int_0^T\int_{\T^d}u_i^N\cdot\diver(\rho_i^N\psi)dxdt \\
  &= \int_0^T\int_{\T^d}\big(2(\rho_i^N)^{1/2}u_i^N\cdot\psi\cdot
  \na(\rho_i^N)^{1/2} + \rho_i^Nu_i^N\cdot\diver\psi\big)dxdt
  \nonumber \\
  &\to \int_0^T\int_{\T^d}\big(2\sqrt{\rho_i}u_i\cdot\psi\cdot
  \na\sqrt{\rho_i} + \rho_iu_i\cdot\diver\psi\big)dxdt \nonumber \\
  &= -\int_0^T\langle\na u_i,\rho_i\psi\rangle_{H^1(\T^d)',
  H^1(\T^d)}dt = -\int_0^T\langle\rho_i\na u_i,\psi\rangle_{
  \mathcal{D}'(\T^d),\mathcal{D}(\T^d)}dt. \nonumber
\end{align}
Thus, $\diver(\rho_i^N\na u_i^N)\to\diver(\rho_i\na u_i)$ in the sense of distributions. In fact, because of the uniform bounds of $(\rho_i^N)^{1/2}$ in $L^\infty(0,T;H^1(\T^d))\hookrightarrow L^\infty(0,T;L^6(\T^d))$ and of $(\rho_i^N)^{1/2}\na u_i^N$ in $L^2(0,T;L^2(\T^d))$, this convergence also holds in the weak topology of $L^2(0,T;L^{3/2}(\T^d))$. The same bounds imply that
\begin{align*}
  u_i^N\otimes\na\rho_i^N = 2(\rho_i^N)^{1/2}u_i^N\otimes\na
  (\rho_i^N)^{1/2}
  \rightharpoonup 2\sqrt{\rho_i}u_i\otimes\na\sqrt{\rho_i}
  = u_i\otimes\na\rho_i
\end{align*}
weakly in $L^1(0,T;L^1(\T^d))$ and hence $\diver(u_i^N\otimes\na\rho_i^N)\rightharpoonup\diver(u_i\otimes\na\rho_i)$ weakly in $L^1(0,T;$ $H^s(\T^d)')$.

Next, since $\Delta(\rho_i^N)^{1/2}\rightharpoonup\Delta\sqrt{\rho_i}$ weakly in $L^2(0,T;L^2(\T^d))$ and $\na(\rho_i^N)^{1/2}\to\na\sqrt{\rho_i}$ strongly in $L^2(0,T;L^2(\T^d))$, for $\psi\in C_0^\infty(\T^d\times(0,T);\R^d)$,
\begin{align*}
  \int_0^T\int_{\T^d}&\rho_i^N\na\bigg(\frac{\Delta(\rho_i^N)^{1/2}}{
  (\rho_i^N)^{1/2}}\bigg)\cdot\psi dxdt \\
  &= -\int_0^T\int_{\T^d}\Delta(\rho_i^N)^{1/2}\big(2\na(\rho_i^N)^{1/2}
  \cdot\phi + (\rho_i^N)^{1/2}\diver\psi\big)dxdt \\
  &\to -\int_0^T\int_{\T^d}\Delta\sqrt{\rho_i}\big(2\na\sqrt{\rho_i}
  \cdot\psi + \sqrt{\rho_i}\diver\psi\big)dxdt.
\end{align*}

The convergence $\rho_i^N|u_i^N|^2u_i^N\to \rho_i|u_i|^2u_i$ strongly in $L^1(0,T;L^1(\T^d))$ has been proved in \cite[Lemma 2.3]{VaYu16}. For completeness, we recall the proof. The strong convergences of $(\rho_i^N)$ and $(\rho_i^Nu_i^N)$ imply, up to subsequence, that
$\rho_i^N\to \rho_i$ and $\rho_i^Nu_i^N\to \rho_iu_i$ a.e. Hence, for a.e.\ $(x,t)$, $u_i^N=(\rho_i^Nu_i^N)/\rho_i^N\to u_i$ whenever $\rho_i^N(x,t)\neq 0$. For a.e.\ $(x,t)$ for which $\rho_i^N(x,t)=0$, we have
\begin{align*}
  g_i^N := \rho_i^N|u_i^N|^2u_i^N\mathrm{1}_{\{|u_i^N|\le M\}}
  \le \rho_I^NM^3 = 0
\end{align*}
for any $M>0$. Consequently, $g_i^N\to \rho_i|u_i|^2u_i\mathrm{1}_{\{|u_i|\le M\}}$ a.e. As the sequence $(\rho_i^N)$ is bounded in $L^\infty(0,T;L^2(\T^d))$, $g_i^N$ is bounded in the same space. Then dominated convergence implies that
\begin{align}\label{2.giN}
  g_i^N\to \rho_i|u_i|^2u_i\mathrm{1}_{\{|u_i|\le M\}}
  \quad\mbox{strongly in }L^1(0,T;L^1(\T^d)).
\end{align}
Now, for any $M>0$,
\begin{align*}
  \int_0^T&\int_{\T^d}\big|\rho_i^N|u_i^N|^2u_i^N
  - \rho|u_i|^2u_i\big|dxdt \\
  &\le \int_0^T\int_{\T^d}
  \big|\rho_i^N|u_i^N|^2u_i^N\mathrm{1}_{\{|u_i^N|\le M\}}
  - \rho_i|u_i|^2u_i\mathrm{1}_{\{|u_i|\le M\}}\big|dxdt \\
  &\phantom{xx}+ \int_0^T\int_{\T^d}
  \big(\rho_i^N|u_i^N|^3\mathrm{1}_{\{|u_i^N|>M\}}
  + \rho_i|u_i|^3\mathrm{1}_{\{|u_i|>M\}}\big)dxdt \\
  &\le \int_0^T\int_{\T^d}
  \big|g_i^N - \rho_i|u_i|^2u_i
  \mathrm{1}_{\{|u_i|\le M\}}\big|dxdt \\
  &\phantom{xx}+ \frac{1}{M}\int_0^T\int_{\T^d}\big(
  \rho_i^N|u_i^N|^4 + \rho_i|u_i|^4\big)dxdt,
\end{align*}
observing that $\rho_i|u_i|^4$ is an element of $L^1(0,T;L^1(\T^d))$. The convergence \eqref{2.giN} shows that
\begin{align*}
  \lim_{N\to\infty}\int_0^T\int_{\T^d}\big|\rho_i^N|u_i^N|^2u_i^N
  - \rho|u_i|^2u_i\big|dxdt \le \frac{C}{M}
\end{align*}
for some $C>0$ and for all $M>0$. The limit $M\to\infty$ finishes the proof of the strong convergence of $\rho_i^N|u_i^N|^2u_i^N$.

The weak convergence of $\na(\rho_1^N+\rho_2^N)$ and the strong convergence of $\rho_i^N$ in $L^2(0,T;L^2(\T^d))$ imply that
\begin{align*}
  \rho_i^N\na(\rho_1^N+\rho_2^N)\rightharpoonup\rho_i\na(\rho_1+\rho_2)
  \quad\mbox{weakly in }L^1(0,T;L^1(\T^d)).
\end{align*}
We treat the $\delta$-regularized terms:
\begin{align*}
  \bigg|\delta\int_0^T\int_{\T^d}\na\rho_i^N\cdot\na\phi dxt\bigg|
  &\le \delta\|\rho_i^N\|_{L^1(0,T;L^1(\T^d))}
  \|\Delta\phi\|_{L^\infty(0,T;L^\infty(\T^d))} \to 0, \\
  \bigg|\delta\eps\int_0^T\int_{\T^d}(u_i^N\otimes\na\rho_i^N)
  :\na\phi dxdt\bigg| &\le 2\delta\eps\|(\rho_i^N)^{1/2}u_i^N
  \|_{L^2(0,T;L^2(\T^d))} \\
  &\phantom{xx}\times\|\na(\rho_i^N)^{1/2}\|_{L^2(0,T;L^2(\T^d))}
  \|\na\phi\|_{L^\infty(0,T;L^\infty(\T^d))} \to 0
\end{align*}
as $N\to\infty$ and $\delta\to 0$. These convergences are sufficient to pass to the limit $N\to\infty$ and $\delta\to 0$ in \eqref{2.inteq}.

It remains to perform the limit in the approximate energy inequality (see Lemma \ref{lem.EN}). This follows from the weak lower semicontinuity of the norms if we show that
\begin{align*}
  (\rho_i^N)^{1/2}\na u_i^N\rightharpoonup\sqrt{\rho_i}\na u_i
  &\quad\mbox{weakly in }L^2(0,T;L^2(\T^d)), \\
  (\rho_i^N)^{1/4}u_i^N \rightharpoonup \sqrt[4]{\rho_i}u_i
  &\quad\mbox{weakly in }L^4(0,T;L^4(\T^d)).
\end{align*}
In fact, we deduce from the bound on $(\rho_i^N)^{1/4}u_i^N$ that
$(\rho_i^N)^{1/4}u_i^N\rightharpoonup y_i$ weakly in $L^4(0,T;$ $L^4(\T^d))$ for some $y_i$. The strong convergence of $(\rho_i^N)^{1/4}$ in $L^4(0,T;L^4(\T^d))$ and the weak convergence of $u_i^N$ in $L^2(0,T;L^2(\T^d))$ imply that $(\rho_i^N)^{1/4}u_i^N\rightharpoonup\sqrt[4]{\rho_i}u_i$ weakly in $L^{4/3}(0,T;$ $L^{4/3}(\T^d))$ and consequently $y_i=\sqrt[4]{\rho_i}u_i$.
The remaining limit has been shown in \eqref{rhonau}.


\section{Proof of Theorem \ref{thm.eps}}\label{sec.eps}

We want to pass to the limit $\eps\to 0$ in \eqref{1.mass}--\eqref{1.mom}. Let $(\rho_i^\eps,u_i^\eps)_{i=1,2}$ be a weak solution to \eqref{1.mass}--\eqref{1.mom} constructed in Theorem \ref{thm.ex}. We set $\rho^\eps=(\rho_1^\eps,\rho_2^\eps)$, $u^\eps=(u_1^\eps,u_2^\eps)$, and $\bar{\rho}^\eps=\rho_1^\eps+\rho_2^\eps$. The energy and entropy estimates from Corollaries \ref{lem.est1} and \ref{lem.est2} yield in the limit $N\to\infty$ and $\delta\to 0$ the existence of a constant $C>0$ independent of $\eps$ such that for $i=1,2$,
\begin{equation}\label{3.unifeps}
\begin{aligned}
  \|\rho_i^\eps\|_{L^\infty(0,T;L^2(\T^d))}
  + \|\bar\rho^\eps\|_{L^2(0,T;H^1(\T^d))}
  + \|(\rho_i^\eps)^{1/2}u_i^\eps\|_{L^2(0,T;L^2(\T^d))} &\le C,
  \\ 
  \sqrt{\eps}\|(\rho_i^\eps)^{1/2}\|_{L^2(0,T;H^2(\T^d))}
  + \sqrt[4]{\eps}\|(\rho_i^\eps)^{1/4}\|_{L^4(0,T;W^{1,4}(\T^d))}
  &\le C, \\ 
  \sqrt{\eps}\|u_i^\eps\|_{L^2(0,T;L^2(\T^d))}
  + \sqrt[4]{\eps}\|(\rho_i^\eps)^{1/4}u_i^\eps\|_{L^4(0,T;L^4(\T^d))}
  &\le C.
\end{aligned}
\end{equation}
This shows that $\rho_i^\eps u_i^\eps = (\rho_i^\eps)^{1/2}\cdot(\rho_i^\eps)^{1/2}u_i^\eps$ is uniformly bounded in $L^2(0,T;L^{4/3}(\T^d))$. Therefore, there exist subsequences (not relabeled) such that, as $\eps\to 0$,
\begin{align*}
  \rho_i^\eps\rightharpoonup\rho_i &\quad\mbox{weakly* in }
  L^\infty(0,T;L^2(\T^d)), \\
  \rho_i^\eps u_i^\eps \rightharpoonup J_i &\quad\mbox{weakly in }
  L^2(0,T;L^{4/3}(\T^d)),\ i=1,2.
\end{align*}
In particular, $\pa_t\bar\rho^\eps = -\diver(\rho_1^\eps u_1^\eps + \rho_2^\eps u_2^\eps)$ is uniformly bounded in $L^2(0,T;W^{1,4}(\T^d)')$. Thanks to the $L^2(0,T;H^1(\T^d))$ bound for $\bar\rho^\eps$, we can apply the Aubin--Lions compactness lemma to conclude that, up to a subsequence,
\begin{align*}
  \bar\rho^\eps\to\bar\rho \quad\mbox{strongly in }L^2(0,T;L^p(\T^d)),
  \ p<6.
\end{align*}
Then the limit in the sum of \eqref{1.mass} over $i=1,2$, namely
$\pa_t\bar\rho^\eps+\diver(\rho_1^\eps u_1^\eps + \rho_2^\eps u_2^\eps)=0$, shows that $\bar\rho$ solves
\begin{align}\label{3.masslim}
  \pa_t\bar\rho + \diver(J_1+J_2) = 0.
\end{align}

The limit in \eqref{1.mom} is more involved, and we perform the limit only in the sum over $i=1,2$. We treat the various expressions term by term. First, the sequence
\begin{align*}
  \rho_i^\eps u_i^\eps\otimes u_i^\eps
  = \big((\rho_i^\eps)^{1/2}u_i^\eps\big)\otimes
  \big((\rho_i^\eps)^{1/2}u_i^\eps\big)
\end{align*}
is bounded in $L^1(0,T;L^1(\T^d))$, which implies that
\begin{align*}
  \eps\diver(\rho_i^\eps u_i^\eps\otimes u_i^\eps)\rightharpoonup 0
  \quad\mbox{weakly in }L^1(0,T;H^s(\T^d)').
\end{align*}
Furthermore, again by \eqref{3.unifeps},
\begin{align*}
  \sqrt{\eps}\rho_i^\eps\na u_i^\eps
  = \sqrt{\eps}\na (\rho_i^\eps u_i^\eps)
  - 2\sqrt{\eps}((\rho_i^\eps)^{1/2}u_i^\eps) \otimes \na(\rho_i^\eps)^{1/2}
\end{align*}
is uniformly bounded in $L^2(0,T;W^{1,4}(\T^d)')+L^1(0,T;L^1(\T^d))$, and thus
\begin{align*}
  \eps\diver(\rho_i^\eps\na u_i^\eps)\rightharpoonup 0\quad
  \mbox{weakly in }L^1(0,T;W^{2,4}(\T^d)').
\end{align*}
The Korteweg regularization in its weak formulation is estimated as follows:
\begin{align*}
  \bigg|\eps&\int_0^T\int_{\T^d}\Delta(\rho_i^\eps)^{1/2}\big(
  4(\rho_i^\eps)^{1/4}\na(\rho_i^\eps)^{1/4}\cdot\psi
  + (\rho_i^\eps)^{1/2}\diver\psi\big)dxdt\bigg| 
  \\
  &\le \sqrt[4]{\eps}\cdot
  \sqrt{\eps}\|\Delta(\rho_i^\eps)^{1/2}\|_{L^2(0,T;L^2(\T^d))}
  \big(4\sqrt[4]{\eps}\|\na(\rho_i^\eps)^{1/4}\|_{L^4(0,T;L^4(\T^d))}
  \\
  &\phantom{xx}\times\|(\rho_i^\eps)^{1/4}\|_{L^\infty(0,T;L^8(\T^d))}
  \|\psi\|_{L^4(0,T;L^8(\T^d))} \\
  &\phantom{xx}+ \sqrt[4]{\eps}
  \|(\rho_i^\eps)^{1/2}\|_{L^\infty(0,T;L^4(\T^d))}
  \|\diver\psi\|_{L^2(0,T;L^4(\T^d))}\big)
    \le \sqrt[4]{\eps}C\to 0,
\end{align*}
where $\psi\in L^4(0,T;W^{1,4}(\T^d;\R^d))$. The drag forces also vanish in the limit since $\eps u_i^\eps \to 0$ strongly in $L^2(0,T;L^2(\T^d))$ and
\begin{align*}
  \eps\rho_i^\eps|u_i^\eps|^2 u_i^\eps
  = \sqrt[4]{\eps}(\rho_i^\eps)^{1/4}\big(\sqrt[4]{\eps}
  (\rho_i^\eps)^{1/4}|u_i^\eps|\big)^2\big(\sqrt[4]{\eps}
  (\rho_i^\eps)^{1/4}u_i^\eps) = O(\sqrt[4]{\eps}) \to 0
\end{align*}
strongly in $L^{4/3}(0,T;L^{8/7}(\T^d))$. Finally,
\begin{align*}
  \sum_{i=1}^2\rho_i^\eps\na\bar\rho^\eps
  = \bar\rho^\eps\na\bar\rho^\eps
  \rightharpoonup \bar\rho\na\bar\rho
  \quad\mbox{weakly in }L^1(0,T;L^{1}(\T^d)).
\end{align*}

Hence, passing to the limit $\eps\to 0$ in sum of the weak formulation \eqref{1.momw} over $i=1,2$, we find that
$$
  k_1^{-1}J_1 + k_2^{-1}J_2 = -\bar\rho\na\bar\rho,
$$
or $J_1 = -k_1\bar\rho\na\bar\rho - (k_1/k_2)J_2$. Together with \eqref{3.masslim}, $\bar\rho$ solves
\begin{align*}
  \pa_t\bar\rho = k_1\diver(\bar\rho\na\bar\rho)
  - \bigg(1-\frac{k_1}{k_2}\bigg)\diver J_2
  = \diver\bigg(k_1\bar\rho\na\bar\rho
  - \bigg(1-\frac{k_1}{k_2}\bigg)J_2\bigg).
\end{align*}

\begin{remark}\rm
If we can identify $J_2=-k_2\rho_2\na\bar\rho$, we obtain
\begin{align*}
  k_1\bar\rho \na\bar\rho
  + \bigg(1-\frac{k_1}{k_2}\bigg)k_2\rho_2\na\bar\rho
  = (k_1\rho_1+k_2\rho_2)\na\bar\rho,
\end{align*}
such that the limit equation becomes
\begin{align*}
  \pa_t\bar\rho + \diver\big((k_1\rho_1+k_2\rho_2)\na\bar\rho\big) = 0.
\end{align*}
\end{remark}


\section{Proof of Theorem \ref{thm.eps2}: relative entropy inequality}
\label{sec.eps2}

The solution $(\bar\rho,\bar{u})$ with $\bar{u}=-\na\bar\rho$ is shown to solve fluid-type equations for which we derive an associated energy equality. Then we derive an inequality for the difference of the mass and momentum balance equations for $(\rho^\eps,u^\eps)$ and $(\bar\rho,\bar{u})$. These results allow us to prove the relative entropy inequality and to finish the proof of Theorem \ref{thm.eps2}.

\subsection{Energy equality for the limit system}

We analyze the limit system for $\bar\rho=\rho_1+\rho_2$ and $\bar{u}=-\na\bar\rho$. The mass balance equation reads as
\begin{align}\label{3.mass}
  \pa_t\bar\rho + \diver(\bar\rho\bar{u}) = 0,
\end{align}
with initial condition $\bar\rho(0)=\bar\rho^0:=\rho_1^0+\rho_2^0$. Inserting \eqref{3.mass}, the analog of the momentum balance equation for $\bar{u}$ becomes
\begin{align}\label{3.mom}
  \pa_t(\bar\rho\bar{u}) + \diver(\bar\rho\bar{u}\otimes\bar{u})
  = \bar\rho(\pa_t\bar{u}+(\bar{u}\cdot\na)\bar{u})
  = \bar\rho\bar{e},
\end{align}
where $\bar{e}:=\pa_t\bar{u}+(\bar{u}\cdot\na)\bar{u}$ is the material derivative of $\bar{u}$. The initial condition is $\bar\rho(0)\bar{u}(0)=-\bar\rho^0\na\bar\rho^0$.
According to Lemma \ref{lem.weak}, the weak formulations of \eqref{3.mass}--\eqref{3.mom} can be written as
\begin{align}
  0 &= -\int_0^t\int_{\T^d}\bar\rho\pa_s\phi dxds
  + \int_{\T^d}\bar\rho\phi\Big|^{s=t}_{s=0} dx
  + \int_0^t\int_{\T^d}\bar\rho\na\bar\rho\cdot\na\phi dxds,
  \label{3.massw} \\
  0 &= -\eps\int_0^t\int_{\T^d}\bar\rho\bar{u}\cdot\pa_s\psi dxds
  + \eps\int_{\T^d}\bar\rho\bar{u}\cdot\psi\Big|^{s=t}_{s=0} dx
  \label{3.momw}\\
  &\phantom{xx}
  - \eps\int_0^t\int_{\T^d}\bar\rho(\bar{u}\otimes\bar{u}):\na\psi dxds
  - \eps\int_0^t\int_{\T^d}\bar\rho\bar{e}\cdot\psi dxds
  \nonumber 
\end{align}
for test functions $\phi\in C_0^\infty(\T^d\times[0,t])$ and $\psi\in C_0^\infty(\T^d\times[0,t];\R^d)$.

We derive the energy equality associated to \eqref{3.mass}--\eqref{3.mom}. For this, we use the test function $\phi=\bar\rho-\eps|\bar{u}|^2/2$ in \eqref{3.massw} and recall that $\na\bar\rho=-\bar{u}$:
\begin{align*}
  0 &= -\int_0^t\int_{\T^d}\bar\rho\pa_s\bigg(\bar\rho
  - \frac{\eps}{2}|\bar{u}|^2\bigg) dxds
  + \int_{\T^d}\bar\rho(t)\bigg(\bar\rho(t) - \frac{\eps}{2}
  |\bar{u}(t)|^2\bigg)dx \\
  &\phantom{xx}- \int_{\T^d}\bar\rho^0\bigg(\bar\rho^0
  - \frac{\eps}{2}|\bar{u}^0|^2\bigg) dx
  + \int_0^t\int_{\T^d}\bar\rho|\bar{u}|^2 dx ds
  + \eps\int_0^t\int_{\T^d}\bar\rho\bar{u}\cdot\na\bar{u}\cdot\bar{u}
  dxds.
\end{align*}
Furthermore, with the test function $\psi=\bar{u}$ in \eqref{3.momw}, we find that
\begin{align*}
  0 &= -\eps\int_0^t\int_{\T^d}\bar\rho\bar{u}\cdot\pa_s\bar{u} dxds
  + \eps\int_{\T^d}\bar\rho(t)|\bar{u}(t)|^2 dx
  - \eps\int_{\T^d}\bar\rho^0|\bar{u}^0|^2 dx \\
  &\phantom{xx}
  - \eps\int_0^t\int_{\T^d}\bar\rho\bar{u}\cdot\na\bar{u}\cdot\bar{u}dxds
  - \eps\int_0^t\int_{\T^d}\bar\rho\bar{e}\cdot\bar{u}dxds.
\end{align*}
Adding the last two equations, some terms cancel, and observing that
$$
  \int_0^t\int_{\T^d}\bar\rho\pa_s\bar\rho dxds
  = \frac12\int_{\T^d}\bar\rho^2\Big|^{s=t}_{s=0}dx
$$
leads to the energy equality
\begin{align}\label{3.E0}
  E_0(\bar\rho(t),\bar{u}(t)) - E_0(\bar\rho^0,\bar{u}^0)
  + \int_0^t\int_{\T^d}\bar\rho|\bar{u}|^2 dxds
  = \eps\int_0^t\int_{\T^d}\bar\rho\bar{e}\cdot\bar{u} dxds,
\end{align}
where the energy associated to \eqref{3.mass}--\eqref{3.mom} is defined by
\begin{align*}
  E_0(\bar\rho,\bar{u}) = \int_{\T^d}\bigg(\frac{\bar\rho^2}{2}
  + \frac{\eps}{2}\bar\rho|\bar{u}|^2\bigg)dx.
\end{align*}

\subsection{Difference of mass balance equations}

We subtract the energy equality \eqref{3.E0} from the energy inequality \eqref{1.Eineq} (recall that $k_1=k_2=1$):
\begin{align}\label{3.Q3}
  E(&\rho^\eps(t),u^\eps(t)-E_0(\bar\rho(t),\bar{u}(t))
  + \int_0^t\int_{\T^d}\bigg(\sum_{i=1}^2\rho_i^\eps|u_i^\eps|^2
  -\bar\rho|\bar{u}|^2\bigg)dxds \\
  &\phantom{xx}+ \eps\sum_{i=1}^2\int_0^t\int_{\T^d}
  \big(\rho_i^\eps|\na u_i^\eps|^2
  + |u_i^\eps|^2 + \rho_i^\eps|u_i^\eps|^4\big)dxds
  + \eps\int_0^t\int_{\T^d}\bar\rho\bar{e}\cdot\bar{u}dxds \nonumber \\
  &\le E(\rho^0,u^0) - E_0(\bar\rho^0,\bar{u}^0). \nonumber
\end{align}
We use the test function $\phi=\bar\rho-\eps|\bar{u}|^2/2$ in the sum of the mass balance equations \eqref{1.massw} over $i=1,2$ (recalling that $\bar\rho^\eps=\rho_1^\eps+\rho_2^\eps$) and subtract the weak formulation \eqref{3.massw} with the same test function (observing that the terms at $t=0$ cancel):
\begin{align}\label{3.Q1}
  0 &= -\int_0^t\int_{\T^d}(\bar\rho^\eps-\bar\rho)
  \pa_s\bigg(\bar\rho - \frac{\eps}{2}|\bar{u}|^2\bigg)dxds
  + \int_{\T^d}(\bar\rho^\eps-\bar\rho)(t)
  \bigg(\bar\rho(t) - \frac{\eps}{2}|\bar{u}(t)|^2\bigg)dx \\
  &\phantom{xx}- \int_0^t\int_{\T^d}\bigg(\sum_{i=1}^2
  \rho_i^\eps u_i^\eps - \bar\rho\bar{u}\bigg)\cdot\na
  \bigg(\bar\rho - \frac{\eps}{2}|\bar{u}|^2\bigg)dxds.
  \nonumber
\end{align}

\subsection{Difference of momentum balance equations}

We add the weak formulation \eqref{1.momw} with test function $\psi=\bar{u}_i$ for $i=1,2$ and subtract the weak formulation \eqref{3.momw} for the limit system with the same test function (recalling that $(\bar\rho,\bar{u})$ is assumed to be smooth):
\begin{align}\label{3.Q2}
  0 &= -\eps\int_0^t\int_{\T^d}\bigg(\sum_{i=1}^2\rho_i^\eps u_i^\eps
  - \bar\rho\bar{u}\bigg)\cdot\pa_s\bar{u} dxds
  + \eps\int_{\T^d}\bigg(\sum_{i=1}^2\rho_i^\eps u_i^\eps
  - \bar\rho\bar{u}\bigg)\cdot \bar{u}\bigg|^{s=t}_{s=0}dx \\
  &\phantom{xx}
  -\eps\int_0^t\int_{\T^d}\bigg(\sum_{i=1}^2\rho_i^\eps u_i^\eps
  \otimes u_i^\eps - \bar\rho\bar{u}\otimes\bar{u}\bigg):\na\bar{u}dxds
  \nonumber \\
  &\phantom{xx}+ \eps\int_0^t\int_{\T^d}\sum_{i=1}^2
  \big\{\rho_i^\eps\na u_i^\eps:\na\bar{u}
  + \Delta(\rho_i^\eps)^{1/2}\big(2\na(\rho_i^\eps)^{1/2}\cdot\bar{u}
  + (\rho_i^\eps)^{1/2}\diver\bar{u}\big)\big\}dxds \nonumber \\
  &\phantom{xx} + \eps\int_0^t\int_{\T^d}\sum_{i=1}^2
  \big(u_i^\eps\cdot\bar{u}
  + \rho_i^\eps|u_i^\eps|^2 u_i^\eps\cdot\bar{u}\big)dxds
  + \int_0^t\int_{\T^d}\bigg(\sum_{i=1}^2\rho_i^\eps u_i^\eps
  - \bar\rho\bar{u}\bigg)\cdot\bar{u} dxds \nonumber \\
  &\phantom{xx}
  + \int_0^t\int_{\T^d}\big(\bar\rho^\eps\na\bar\rho^\eps
  - \bar\rho\na\bar\rho\big)\cdot\bar{u} dxds
  + \eps\int_0^t\int_{\T^d}\bar\rho\bar{e}\cdot\bar{u}dxds. \nonumber
\end{align}

\subsection{Relative energy inequality}

We recall the definition of the relative energy:
\begin{align*}
  E_R(\rho^\eps,u^\eps|\bar\rho,\bar{u})
  = \int_{\T^d}\bigg\{\frac12(\bar\rho^\eps-\bar\rho)^2
  + \eps\sum_{i=1}^2\bigg(\frac{\rho_i^\eps}{2}|u_i^\eps-\bar{u}|^2
  + |\na(\rho_i^\eps)^{1/2}|^2\bigg)\bigg\}dx,
\end{align*}
where $\bar\rho^\eps:=\rho_1^\eps+\rho_2^\eps$. The main task is to derive the relative entropy inequality.

\begin{lemma}[Relative energy inequality]\label{lem.rel}
There exists a constant $C>0$ independent of $\eps$ such that
\begin{align*}
  E_R(&\rho^\eps(t),u^\eps(t)|\bar\rho(t),\bar{u}(t))
  + \int_0^t\int_{\T^d}\sum_{i=1}^2\rho_i^\eps|u_i^\eps-\bar{u}|^2
  dxds \\
  &\phantom{xx}+ \frac{\eps}{2}\int_0^t\int_{\T^d}\sum_{i=1}^2\big(
  \rho_i^\eps|\na(u_i^\eps-\bar{u})|^2 + |u_i^\eps-\bar{u}|^2
  + \rho_i^\eps|u_i^\eps|^2|u_i^\eps-\bar{u}|^2\big)dxds \\
  &\le C\sqrt[4]{\eps}
  + C\int_0^tE_R(\rho^\eps,u^\eps|\bar\rho,\bar{u})(s)ds.
\end{align*}
In particular, for $i=1,2$, as $\eps\to 0$,
\begin{align*}
  (\rho_i^\eps)^{1/2}(u_i^\eps-\bar{u})\to 0
  \quad\mbox{strongly in }L^2(0,T;L^2(\T^d)).
\end{align*}
\end{lemma}

\begin{proof}
We subtract \eqref{3.Q1} and \eqref{3.Q2} from \eqref{3.Q3}. An elementary computation shows that
\begin{align}\label{3.J1J11}
  \int_{\T^d}&\bigg(\frac12(\bar\rho^\eps-\bar\rho)^2
  + \frac{\eps}{2}\sum_{i=1}^2\rho_i^\eps|u_i^\eps-\bar{u}|^2
  + \eps\sum_{i=1}^2|\na(\rho_i^\eps)^{1/2}|^2\bigg)\Big|^{s=t}_{s=0}dx
  = J_1+\cdots+J_{11},
\end{align}
where
\begin{align*}
  J_1 &= -\int_0^t\int_{\T^d}\sum_{i=1}^2\rho_i^\eps
  |u_i^\eps-\bar{u}|^2dxds, \\
  J_2 &= -\int_0^t\int_{\T^d}\bigg(\sum_{i=1}^2\rho_i^\eps u_i^\eps
  - \bar\rho^\eps\bar{u}\bigg)\cdot\bar{u} dxds, \\
  J_3 &= -\int_0^t\int_{\T^d}(\bar\rho^\eps-\bar\rho)\pa_s
  \bigg(\bar\rho - \frac{\eps}{2}|\bar{u}|^2\bigg)dxds, \\
  J_4 &= -\eps\int_0^t\int_{\T^d}\bigg(\sum_{i=1}^2\rho_i^\eps u_i^\eps
  - \bar\rho\bar{u}\bigg)\cdot\pa_s\bar{u}dxds, \\
  J_5 &= -\int_0^t\int_{\T^d}\bigg(\sum_{i=1}^2\rho_i^\eps u_i^\eps
  - \bar\rho\bar{u}\bigg)\cdot\na
  \bigg(\bar\rho - \frac{\eps}{2}|\bar{u}|^2\bigg)dxds, \\
  J_6 &= -\eps\int_0^t\int_{\T^d}\bigg(\sum_{i=1}^2
  \rho_i^\eps u_i^\eps\otimes u_i^\eps - \bar\rho\bar{u}\otimes\bar{u}
  \bigg):\na\bar{u}dxds, \\
  J_7 &= \int_0^t\int_{\T^d}(\bar\rho^\eps\na\bar\rho^\eps
  - \bar\rho\na\bar\rho)\cdot\bar{u}dxds, \\
  J_8 &= -\eps\int_0^t\int_{\T^d}\sum_{i=1}^2\rho_i^\eps\na u_i^\eps
  :\na(u_i^\eps-\bar{u})dxds, \\
  J_9 &= -\eps\int_0^t\int_{\T^d}\sum_{i=1}^2 u_i^\eps\cdot
  (u_i^\eps-\bar{u})dxds, \\
  J_{10} &= -\eps\int_0^t\int_{\T^d}\sum_{i=1}^2\rho_i^\eps
  |u_i^\eps|^2 u_i^\eps\cdot(u_i^\eps-\bar{u})dxds, \\
  J_{11} &= \eps\int_0^t\int_{\T^d}\Delta(\rho_i^\eps)^{1/2}
  \big(2\na(\rho_i^\eps)^{1/2}\cdot\bar{u} + (\rho_i^\eps)^{1/2}
  \diver\bar{u}\big)dxds.
\end{align*}
The expression on the left-hand side of \eqref{3.J1J11} at $t=0$ is of order $\eps$, because of Assumption (A2) on the initial data and $(\rho^\eps-\bar\rho)(0)=0$. We wish to estimate $J_1,\ldots,J_{11}$.

We split the sum $J_3+\cdots+J_6$ into two parts, $J_3+\cdots+J_6=K_1+K_2$, where
\begin{align*}
  K_1 &= \frac{\eps}{2}\int_0^t\int_{\T^d}(\bar\rho^\eps-\bar\rho)
  \pa_s|\bar{u}|^2 dxds
  - \eps\int_0^t\int_{\T^d}\bigg(\sum_{i=1}^2\rho_i^\eps u_i^\eps
  - \bar\rho\bar{u}\bigg)\cdot\pa_s\bar{u}dxds \\
  &\phantom{xx}+ \frac{\eps}{2}\int_0^t\int_{\T^d}
  \bigg(\sum_{i=1}^2\rho_i^\eps u_i^\eps - \bar\rho\bar{u}\bigg)
  \na|\bar{u}|^2 dxds \\
  &\phantom{xx}
  - \eps\int_0^t\int_{\T^d}\bigg(\sum_{i=1}^2\rho_i^\eps u_i^\eps
  \otimes u_i^\eps - \bar\rho\bar{u}\otimes\bar{u}\bigg)
  :\na\bar{u} dxds \\
  &=: K_{11}+\cdots K_{14}, \\
  K_2 &= -\int_0^t\int_{\T^d}(\bar\rho^\eps-\bar\rho)\pa_s\bar\rho dxds
  - \int_0^t\int_{\T^d}\bigg(\sum_{i=1}^2\rho_i^\eps u_i^\eps
  - \bar\rho\bar{u}\bigg)\cdot\na\bar\rho dxds.
\end{align*}
Some terms cancel in $K_{11}+K_{12}$, and we end up with
\begin{align*}
  K_{11}+K_{12} = -\eps\int_0^t\int_{\T^d}\bigg(\sum_{i=1}^2
  \rho_i^\eps u_i^\eps - \bar\rho^\eps\bar{u}\bigg)\cdot\pa_s\bar{u}dxds.
\end{align*}
Also in the sum $K_{13}+K_{14}$, some terms cancel:
\begin{align*}
  K_{13}+K_{14} &= -\eps\int_0^t\int_{\T^d}\bigg(\sum_{i=1}^2
  \rho_i^\eps u_i^\eps - \bar\rho^\eps\bar{u}\bigg)\cdot\na\bar{u}
  \cdot\bar{u}dxds \\
  &\phantom{xx}- \eps\int_0^t\int_{\T^d}\sum_{i=1}^2
  \rho_i^\eps(u_i^\eps-\bar{u})
  \otimes(u_i^\eps-\bar{u}):\na\bar{u}dxds.
\end{align*}
Recalling that $\bar{e}=\pa_t\bar{u}+\bar{u}\cdot\na\bar{u}$, we obtain
\begin{align*}
  K_1 &= -\eps\int_0^t\int_{\T^d}\bigg(\sum_{i=1}^2
  \rho_i^\eps u_i^\eps - \bar\rho^\eps\bar{u}\bigg)\cdot\bar{e} dxds \\
  &\phantom{xx}- \eps\int_0^t\int_{\T^d}\sum_{i=1}^2
  \rho_i^\eps(u_i^\eps-\bar{u})
  \otimes(u_i^\eps-\bar{u}):\na\bar{u}dxds.
\end{align*}
Since $\bar{e}$ and $\bar{u}$ are assumed to be smooth, we can estimate $K_1$ according to
\begin{align*}
  K_1 &\le \eps C(\bar{e})\|\rho_1^\eps u_1^\eps + \rho_2^\eps u_2^\eps
  - \bar\rho^\eps\bar{u}\|_{L^2(0,T;L^1(\T^d))} \\
  &\phantom{xx}+ \eps C(\na\bar{u})\int_0^t\int_{\T^d}\rho_i^\eps
  |u_i^\eps-\bar{u}|^2 dxds \le C\eps,
\end{align*}
where the last step follows from estimates \eqref{3.unifeps}.

For $K_2$, we insert $\pa_t\bar\rho=-\diver(\bar\rho\bar{u})$, use $\na\bar\rho=-\bar{u}$, and integrate by parts:
\begin{align*}
  K_2 &= -\int_0^t\int_{\T^d}\na(\bar\rho^\eps-\bar\rho)
  \cdot(\bar\rho\bar{u})dxds + \int_0^t\int_{\T^d}
  \bigg(\sum_{i=1}^2\rho_i^\eps u_i^\eps - \bar\rho\bar{u}\bigg)
  \cdot\bar{u}dxds.
\end{align*}
We add to this expression
\begin{align*}
  J_2+J_7 &= -\int_0^t\int_{\T^d}\bigg(\sum_{i=1}^2
  \rho_i^\eps u_i^\eps - \bar\rho^\eps\bar{u}\bigg)\cdot\bar{u} dxds
  + \int_0^t\int_{\T^d}(\bar\rho^\eps\na\bar\rho^\eps
  - \bar\rho\na\bar\rho)\cdot\bar{u} dxds,
\end{align*}
which yields, using again $\bar{u}=-\na\bar\rho$ and integration by parts,
\begin{align*}
  K_2 + J_2 + J_7
  &= -\int_0^t\int_{\T^d}\big(\bar\rho(\na\bar\rho^\eps-\na\bar\rho)
  + (\bar\rho\bar{u}-\bar\rho^\eps\bar{u})
  - (\bar\rho^\eps\na\bar\rho^\eps - \bar\rho\na\bar\rho)
  \big)\cdot\bar{u} dxds \\
  &= \int_0^t\int_{\T^d}(\bar\rho^\eps-\bar\rho)
  \na(\bar\rho^\eps-\bar\rho)\cdot\bar{u} dxds
  = -\frac12\int_0^t\int_{\T^d}(\bar\rho^\eps-\bar\rho)^2
  \diver\bar{u}dxds \\
  &\le C(\bar{u})\int_0^t\int_{\T^d}(\bar\rho^\eps-\bar\rho)^2 dxds.
\end{align*}
This shows that
\begin{align*}
  J_2+\cdots+J_7 = K_1 + K_2 + J_2 + J_7
  \le C\eps + C\int_0^t\int_{\T^d}(\bar\rho^\eps-\bar\rho)^2 dxds.
\end{align*}

Next, we estimate $J_8$, using the Young and Cauchy--Schwarz inequalities:
\begin{align*}
  J_8 &= -\eps\int_0^t\int_{\T^d}\sum_{i=1}^2
  \big(\rho_i^\eps|\na(u_i^\eps-\bar{u})|^2
  + \rho_i^\eps\na\bar{u}:\na(u_i^\eps-\bar{u})\big)dxds \\
  &\le -\frac{\eps}{2}\int_0^t\int_{\T^d}\sum_{i=1}^2
  \rho_i^\eps|\na(u_i^\eps-\bar{u})|^2 dxds
  + \eps\sum_{i=1}^2\|\rho_i^\eps\|_{L^\infty(0,T;L^1(\T^d))}
  \|\na\bar{u}\|_{L^2(0,T;L^\infty(\T^d))}^2 \\
  &\le -\frac{\eps}{2}\int_0^t\int_{\T^d}\sum_{i=1}^2
  \rho_i^\eps|\na(u_i^\eps-\bar{u})|^2 dxds + C\eps,
\end{align*}
since the total mass of $\rho_i^\eps$ is uniformly bounded. We estimate $J_9$ and $J_{10}$ in a similar way, leading to
\begin{align*}
  J_9 + J_{10} &\le -\frac{\eps}{2}\int_0^t\int_{\T^d}\sum_{i=1}^2\big(
  |u_i^\eps-\bar{u}|^2 + \rho_i^\eps|u_i^\eps|^2|u_i^\eps-\bar{u}|^2
  \big)dxds + C\eps.
\end{align*}
By the uniform estimates \eqref{3.unifeps},
\begin{align*}
  J_{11} &\le \sqrt[4]{\eps}\sum_{i=1}^2\sqrt {\eps}
  \|\Delta(\rho_i^\eps)^{1/2}\|_{L^2(0,T;L^2(\T^d))}\big(4\sqrt[4]{\eps}
  \|\na(\rho_i^\eps)^{1/4}\|_{L^4(0,T;L^4(\T^d))} \\
  &\phantom{xx}\times\|(\rho_i^\eps)^{1/4}\|_{L^4(0,T;L^4(\T^d))}
  \|\bar{u}\|_{L^\infty(0,T;L^\infty(\T^d))} \\
  &\phantom{xx}+ \|(\rho_i^\eps)^{1/2}\|_{L^\infty(0,T;L^2(\T^d))}
  \|\diver\bar{u}\|_{L^2(0,T;L^\infty(\T^d))}\big)
  \le C(\bar{u})\sqrt[4]{\eps}.
\end{align*}
Thus, we obtain
\begin{align*}
  J_8&+\cdots+J_{11} \\
  &\le C\sqrt[4]{\eps} - \frac{\eps}{2}
  \int_0^t\int_{\T^d}\sum_{i=1}^2\big(
  \rho_i^\eps|\na(u_i^\eps-\bar{u})|^2 + |u_i^\eps-\bar{u}|^2 + \rho_i^\eps|u_i^\eps|^2|u_i^\eps-\bar{u}|^2\big)dxds.
\end{align*}

We summarize the previous estimates to infer from \eqref{3.J1J11} that
\begin{align*}
  \int_{\T^d}&\bigg(\frac12(\rho^\eps-\bar\rho)^2
  + \frac{\eps}{2}\sum_{i=1}^2\rho_i^\eps|u_i^\eps-\bar{u}|^2
  + \eps\sum_{i=1}^2|\na(\rho_i^\eps)^{1/2}|^2\bigg)(t)dx \\
  &\phantom{xx}
  + \int_0^t\int_{\T^d}\sum_{i=1}^2\rho_i^\eps|u_i^\eps-\bar{u}|^2
  dxds \\
  &\phantom{xx}+ \frac{\eps}{2}\int_0^t\int_{\T^d}\sum_{i=1}^2\big(
  \rho_i^\eps|\na(u_i^\eps-\bar{u})|^2 + |u_i^\eps-\bar{u}|^2
  + \rho_i^\eps|u_i^\eps|^2|u_i^\eps-\bar{u}|^2\big)dxds \\
  &\le C\sqrt[4]{\eps} + C\int_0^t\int_{\T^d}
  (\bar\rho^\eps-\bar\rho)^2 dxds.
\end{align*}
Gronwall's lemma concludes the proof of the lemma.
\end{proof}

\subsection{Proof of Theorem \ref{thm.eps2}}

It follows from Lemma \ref{lem.rel} that
\begin{align*}
  \|\rho_i^\eps(u_i^\eps-\bar{u})\|_{L^1(0,T;L^1(\T^d))}
  \le \|(\rho_i^\eps)^{1/2}\|_{L^2(0,T;L^2(\T^d))}
  \|(\rho_i^\eps)^{1/2}(u_i^\eps-\bar{u})\|_{L^2(0,T;L^2(\T^d))}
  \to 0
\end{align*}
as $\eps\to 0$. Together with the convergence $\rho_i^\eps\rightharpoonup \rho_i$ weakly* in $L^\infty(0,T;L^2(\T^d))$, this shows that
\begin{align*}
  \rho_i^\eps u_i^\eps = \rho_i^\eps(u_i^\eps-\bar{u})
  + \rho_i^\eps\bar{u} \rightharpoonup \rho_i\bar{u}
  \quad\mbox{weakly in }L^1(0,T;L^1(\T^d)).
\end{align*}
By definition, $\bar{u}=-\na\bar\rho$. Since $\rho_i^\eps u_i^\eps\rightharpoonup J_i$ weakly in $L^2(0,T;L^{4/3}(\T^d))$, we infer that $J_i=-\rho_i\na\bar\rho$. In particular, $\rho_i$ solves the transport equation $\pa_t\rho_i=-\diver J_i = \diver(\rho_i\na\bar\rho)$, while $\bar\rho$ is the solution to the porous-medium equation $\pa_t\bar\rho=\diver(\bar\rho\na\bar\rho)$ with $\bar\rho(0)=\rho_1^0+\rho_2^0$ in $\T^d$.


\end{document}